\documentclass[12pt, oneside, notitlepage]{amsart}
\usepackage[margin=3cm]{geometry}
\usepackage{amsmath,amssymb,amsthm,graphicx,mathrsfs,bbm,url}
\usepackage{amsthm}
\usepackage{wrapfig}
\usepackage{enumitem}
\usepackage{mathtools}
\usepackage[utf8]{inputenc}
 \usepackage[T1]{fontenc}
\usepackage[usenames,dvipsnames]{color}
\usepackage[colorlinks=true,linkcolor=Red,citecolor=Green]{hyperref}
\usepackage[super]{nth}
\usepackage[open, openlevel=2, depth=3, atend]{bookmark}
\hypersetup{pdfstartview=XYZ}
\usepackage[font=footnotesize]{caption}
\usepackage{epstopdf}
 
\usepackage{hyperref}

\theoremstyle{plain}
\newtheorem{theorem}{Theorem}[section]
\newtheorem*{theorem*}{Theorem}

\newtheorem{lemma}[theorem]{Lemma}

\newtheorem{proposition}[theorem]{Proposition}

\theoremstyle{definition}
\newtheorem{definition}[theorem]{Definition}

\theoremstyle{remark}
\newtheorem{remark}[theorem]{Remark}

\numberwithin{equation}{section}

\newcommand{\R}{\mathbb{R}}

\newcommand{\Z}{\mathbb{Z}}

\newcommand{\N}{\mathbb{N}}

\newcommand{\eps}{\varepsilon}

\newcommand{\mc}{\mathcal}

\newcommand{\la}{\lambda}

\newcommand{\dd}{\mathrm{d}}

\DeclareMathOperator{\vol}{vol}
\DeclareMathOperator{\Ell}{ell}

\DeclareMathOperator{\Op}{Op}

\DeclareMathOperator{\ran}{ran}

\DeclareMathOperator{\supp}{supp}

\DeclareMathOperator{\comp}{c}

\newcommand{\be}{\begin{equation}}
\newcommand{\ee}{\end{equation}}

\newcommand{\ga}{\gamma}
\newcommand{\op}{\operatorname}
\newcommand{\Si}{\Sigma}

\newcommand{\Om}{\Omega} 
\newcommand{\al}{\alpha}

\newcommand{\bigo}{\mathcal{O}}
\newcommand{\con}{\overline}

% ...

% ...

\title
[Semiclassical measures of magnetic Laplacians]
{Semiclassical defect measures of magnetic Laplacians on hyperbolic surfaces}

\author{Laurent Charles}

\address{Université de Paris and Sorbonne Université, CNRS, IMJ-PRG, F-75006 Paris, France.}

\email{lcharles@imj-prg.fr}

\author{Thibault Lefeuvre}

\address{Université de Paris and Sorbonne Université, CNRS, IMJ-PRG, F-75006 Paris, France.}

\email{tlefeuvre@imj-prg.fr}

\begin{document}
%% Résumé

%% Résumé anglais
\begin{abstract}
On a closed hyperbolic surface, we investigate semiclassical defect measures associated with the magnetic Laplacian in the presence of a constant magnetic field. Depending on the energy level where the eigenfunctions concentrate, three distinct dynamical regimes emerge. In the low-energy regime, we show that any invariant measure of the magnetic flow in phase space can be obtained as a semiclassical measure. At the critical energy level, we establish Quantum Unique Ergodicity, together with a quantitative rate of convergence of eigenfunctions to the Liouville measure. In the high-energy regime, we prove a Shnirelman-type result: a density-one subsequence of eigenfunctions becomes equidistributed with respect to the Liouville measure.
\end{abstract}

\maketitle

\section{Introduction}

This article is the first in a two-part series devoted to the study of semiclassical defect measures associated with magnetic Laplacians on surfaces. In this first part, we focus on the case of surfaces with constant curvature under the influence of a constant magnetic field. The second part \cite{Charles-Lefeuvre-25-2} will address the more general setting of surfaces with variable curvature and non-constant magnetic fields. The main result of this article is Theorem \ref{theorem:horocyclic}, which establishes a polynomial rate of convergence of eigenfunctions to the Liouville measure in the critical energy regime.

\subsection{Setting}

Let $(\Sigma,g)$ be a closed connected oriented hyperbolic (constant curvature $-1$) surface of genus $\textsl{g} \geq 2$. Let $L \to \Sigma$ be a Hermitian line bundle equipped with a unitary connection $\nabla$, and denote by $F_\nabla = -i B \vol \in C^\infty(\Sigma, \Lambda^2 T^*\Sigma)$ the curvature $2$-form of $\nabla$, where $B \in C^\infty(\Sigma)$, and $\vol$ is the Riemannian volume. We call $B$ the \emph{magnetic field}. The \emph{magnetic Laplacian} is defined as
\begin{equation}
\label{equation:delta}
\Delta_L := \tfrac{1}{2} \nabla^* \nabla : C^\infty(\Sigma,L) \to C^\infty(\Sigma,L).
\end{equation}
More generally, taking $L^{\otimes k}$ for $k \in \Z_{\geq 0}$ and the induced connection $\nabla^{\otimes k}$, one can form similarly to \eqref{equation:delta} a Laplacian $\Delta_k$ acting on $C^\infty(\Sigma,L^{\otimes k})$. The operator $k^{-2}\Delta_k$ is a twisted semiclassical (pseudo)differential operator, with semiclassical parameter $h := k^{-1} > 0$.

Throughout this article, we will further assume that $B$ is \emph{constant}. Note that, by Gauss-Bonnet, this implies that $2B(\textsl{g}-1) \in \Z$. The purpose of this paper is to study the semiclassical limits of the Laplace eigenstates
\begin{equation}
\label{equation:eigenstates}
 k^{-2}\Delta_k u_k = (E+\eps_k)u_k, \qquad u_k \in C^\infty(\Sigma,L^{\otimes
   k}), \;  \|u_k\|_{L^2(\Sigma,L^{\otimes k})} = 1,
\end{equation}
in the regime $k \to +\infty$, where $E \geq 0$ and $\eps_k \to_{k \to +\infty} 0$. We say that $u_k$ converges to the semiclassical defect measure $\mu$ (defined on $T^*\Sigma$) if for all $a \in C^\infty(T^*\Sigma)$,
\begin{equation}
\label{equation:convergence}
\langle\Op_{k}(a)u_{k},u_{k}\rangle_{L^2} \to_{k \to \infty} \int_{T^*\Sigma} a(x,\xi) \dd\mu(x,\xi),
\end{equation}
where $\Op$ is a certain semiclassical magnetic quantization on $\Sigma$ (see \S\ref{section:microlocal} for a definition). We denote by $u_k \rightharpoonup_{k \to \infty} \mu$ this convergence. By the normalization assumption $\|u_k\|_{L^2(\Sigma,L^{\otimes k})} = 1$, $\mu$ is a probability measure; in addition, $\mu$ must be supported on the compact energy shell $\{p=E\} \subset T^*\Sigma$ (where $p(x,\xi) := \tfrac{1}{2}|\xi|^2$). The limit \eqref{equation:convergence} may not exist as $k \to +\infty$; however, it always exists along a subsequence $(k_n)_{n \geq 0}$ and we will often drop the subsequence notation. We shall see that three different semiclassical regimes appear according to the value of $E$, affecting the possible behaviour of $\mu$. 

\subsection{Main results}

Let $\omega_0$ be the Liouville symplectic $2$-form on $T^*\Sigma$. Set
\[
\Omega := \omega_0 + i \pi^* F_\nabla,
\]
where $\pi : T^*\Sigma \to \Sigma$ is the projection. This is the Liouville $2$-form with a magnetic correction; it is still a symplectic $2$-form on $T^*\Sigma$. The (semiclassical) principal symbol of $k^{-2}\Delta_k$ is $p(x,\xi) = \tfrac{1}{2}|\xi|^2_g$. The Hamiltonian vector field $H_p^\Omega$ of $p$ computed with respect to $\Omega$ is defined through the relation
\[
dp(\bullet) = \Omega(\bullet, H^{\Omega}_p).
\]
The Hamiltonian flow $(\Phi_t)_{t \in \R}$ generated by $H_p^{\Omega}$ is called the \emph{magnetic flow}. As any autonomous Hamiltonian flow, $(\Phi_t)_{t \in \R}$ preserves the (compact) energy layers $\{p=E\}$ and a natural smooth probability measure $\mu_{\mathrm{Liouv}}$ on $\{p=E\}$ called the Liouville measure (see \S\ref{ssection:geo} for these definitions).

Let $S\Sigma \to \Sigma$ be the unit tangent bundle. Denote by $(\varphi_t)_{t \in \R}$ the (Anosov) geodesic flow, $(R_t)_{t \in \R}$ the $2\pi$-periodic rotation in the circle fibers of $S\Sigma$, and $(h_t)_{t \in \R}$ the stable horocyclic flow. Define
\[
E_c := \tfrac{1}{2}B^2, \qquad T_E = (B^2-2E)^{-1/2} \text{ if } E < E_c, \qquad  T_E := (2E-B^2)^{-1/2}  \text{ if } E > E_c.
\]
The energy $E_c$ is called the \emph{critical energy}. The following fact is
well-known \cite{Arnold-61,Ginzburg-96} and will be reproved quickly in
\S\ref{ssection:magnetic-flow}. For $E > 0$, the magnetic flow $(\Phi_t)_{t \in \R}$ on $\{p=E\}
\subset T^*\Sigma$ is conjugate either to:
\begin{enumerate}[label=(\roman*)]
  \item the reparametrized rotation flow $(R_{t/T_E})_{t \in \R}$ of period
    $2\pi \cdot T_E$ if $E < E_c$ (elliptic case),
  \item the horocyclic flow $(h_t)_{t \in \R}$ if $E=E_c$ (parabolic case),
    \item the reparametrized geodesic flow $(\varphi_{t/T_E})_{t \in \R}$ if $E>E_c$ (hyperbolic case).
    \end{enumerate}
Notice that for $E < E_c$ the space of orbits at energy $E$ is diffeomorphic to $\Si$ itself,
so the $(\Phi_t)_{t \in \R}$ invariant probability measures of $\{ p = E\}$
identify with the probability measures of $\Si$. This holds as well for
$E=0$ because the flow is stationnary on $\{ p =0 \}$.

    We will prove that this transition in the dynamics of the magnetic flow at the critical energy $E=E_c$ translates at the quantum level into the following result on semiclassical defect measures:

\begin{theorem}[The three classical/quantum regimes]
\label{theorem:main}
The following holds under the above assumptions on $(\Sigma,g)$ and $(L,\nabla)$:
\begin{enumerate}[label=\emph{(\roman*)}]
\item \emph{\textbf{Low energies regime.}} If $0 \leq E < E_c$, then for any $(\Phi_t)_{t \in \R}$ invariant probability measure $\mu$ on $\{p=E\}$, there exists a sequence $(u_{k})_{k \geq 0}$ satisfying \eqref{equation:eigenstates} such that $u_{k} \rightharpoonup_{k \to \infty} \mu$; 
\item \emph{\textbf{Critical energy regime.}} For any
  sequence $(u_k)_{k \geq 0}$ satisfying  \eqref{equation:eigenstates} with $E=E_c$,  $u_{k} \rightharpoonup_{k \to \infty} \mu_{\mathrm{Liouv}}$
  %the only semiclassical defect measure associated to a sequence $(u_{k})_{k
  %\geq 0}$ satisfying \eqref{equation:eigenstates} is the Liouville measure
  %$\mu = \mu_{\mathrm{Liouv}}$;
\item \emph{\textbf{High energies regime.}} If $E_c < a < b$, consider for
  all $k \geq 0$, the set of eigenstates $(u_{k,j})_{k,j \in A}$ such that
  $k^{-2} \Delta_k u_{k,j} = \lambda_{k,j} u_{k,j}$ with $\lambda_{k,j} \in
  [a,b]$ and $A \subset \Z_{\geq 0}^2$. Then there exists a density one subset
  $A_\star \subset A$ such that for all sequences $(u_{k_n,j_n})_{n \geq 0}$
  with $(k_n,j_n) \in A_\star$ and $k_n \to \infty$,  $u_{k_n,j_n} \rightharpoonup_{n \to \infty} \mu_{\mathrm{Liouv}}$.

\end{enumerate}
\end{theorem}

In item (iii), $A_\star$ is a density one subset of $A \subset \Z^2_{\geq 0}$ if it satisfies:
\[
\dfrac{\sharp\{ (k,j) \in A_\star ~|~ k^2+j^2 \leq T\}}{\sharp\{ (k,j) \in A ~|~ k^2 + j^2 \leq T\}} \to_{T \to +\infty} 1.
\]

Parts (ii) and (iii) of the above result are straightforward observations to
derive and were already mentioned by Zelditch \cite[pp. 20–21]{Zelditch-92}.
In the third regime, corresponding to energies $E > E_c$, the eigenfunctions
of the magnetic Laplacian are in correspondence with those of the usual
Laplace–Beltrami operator on $(\Sigma,g)$, see Proposition
\ref{theo:Landau_Level} below. In this setting, it is conjectured that the
Liouville measure is the unique semiclassical defect measure—this is the
content of the celebrated and widely open Quantum Unique Ergodicity (QUE)
conjecture. Further background can be found in \S\ref{ssection:literature}.
The first regime $E < E_c$ is less obvious and relies on the construction of
appropriate eigenstates concentrating along a given periodic orbit of the
magnetic flow. This is carried out in \S\ref{section:constant} using
Weinstein's averaging method \cite{Weinstein-77}.

The second regime, corresponding to the critical energy level $E = E_c$, provides a particularly intriguing—albeit trivial—manifestation of QUE. Indeed, any defect measure must be invariant under the magnetic Hamiltonian flow, which coincides with the horocyclic flow on this energy shell. As the horocyclic flow is uniquely ergodic (its only invariant measure being the Liouville measure, see \cite{Furstenberg-73}), it follows that Liouville is the only possible defect measure in this case.

Moreover, we can obtain a more precise statement concerning the convergence of eigenfunctions. Let $0 < \theta < 1/2$ be a real number satisfying $\theta(1-\theta) \leq \lambda_1(\Sigma)$, where $\lambda_1(\Sigma)$ denotes the first non-zero eigenvalue of the Laplacian $\Delta_g$ on functions. Let $S^*_c\Sigma$ denote the energy shell $\{p=E_c\}$. Then the following result holds:

\begin{theorem}[Polynomial rate of convergence in the critical regime]
\label{theorem:horocyclic}
Suppose that \eqref{equation:eigenstates} holds with $E=E_c$ and $\eps_k \leq h^{\ell}$ for some $\ell > 0$. Then there exists $C := C(\ell) > 0$ such that for all $a \in C^\infty(T^*\Sigma)$ with support in $\{p \leq 10E_c\}$:
\[
\left|\langle \Op_k(a) u_k,u_k\rangle_{L^2} - \int_{S^*_c\Sigma} a(x,\xi) \dd \mu_{\mathrm{Liouv}}(x,\xi)\right| \leq Ck^{-\theta \min(\ell,1/15)/4100}\|a\|_{C^{17}(T^*\Sigma)}.
\]
\end{theorem}

We emphasize that the remainder term above is very likely far from optimal. We did not attempt to optimize the polynomial exponent in our proof. The function $a$ need not be compactly supported. Instead, one may also consider a general symbol $a \in S^m(T^*\Sigma)$, exploiting the fact that the eigenstates $u_k$ concentrate on the energy shell $\{p=E_c\}$. This affects the convergence rate in Theorem \ref{theorem:horocyclic} only through a negligible remainder term of order $\mc{O}(k^{-\infty})$.

A related quantitative result for QUE was recently obtained by Morin and Rivière \cite{Morin-Riviere-24} in a different regime (the magnetic field is constant and the semiclassical parameter is the inverse square root of the energy), in the context of their study of magnetic Laplacians on the torus when the magnetic field is constant.

Next, we discuss the case where the operator is perturbed by a small potential. Let
\[
D^*\Sigma := \{p < E_c\}.
\]
(This also corresponds to $D^*\Sigma= \{(x,\xi) \in T^*\Sigma ~|~ |\xi| < B\}$.) Given $V \in C^\infty(\Sigma)$, we define $\langle V \rangle \in C^\infty(D^*\Sigma)$ by setting for $(x,\xi) \in D^*\Sigma$
\[
\langle V \rangle(x,\xi) := \dfrac{1}{2\pi \cdot T_E} \int_0^{2\pi \cdot T_E} V(\pi(\Phi_t(x,\xi))) \dd t,
\]
where $\pi : T^*\Sigma \to \Sigma$ is the projection and $(\Phi_t)_{t \in \R}$ is the magnetic flow defined above.

\begin{theorem}[Perturbation by a potential]
\label{theorem:constant2}
Let $\eps > 0$, $V \in C^\infty(\Sigma)$. Let $(u_k)_{k \geq 0}$ be a sequence of eigenfunctions satisfying
\[
(k^{-2}\Delta_k + k^{-2}V) u_k = E_k u_k, \qquad 0 \leq E_k \leq E_c-\eps,\; \|u_k\|_{L^2}=1.
\]
Then any semiclassical limit $u_k \rightharpoonup_{k \to \infty} \mu$ is invariant by the Hamiltonian vector fields $H_p^\Omega$ and $H_{\langle V\rangle}^\Omega$ of the Hamiltonians $p$ and $\langle V\rangle$.
\end{theorem}

More generally, one can also consider a perturbation of the form $k^{-1}V$ instead of $k^{-2}V$, provided that $\|V\|_{L^\infty}$ is small enough depending on $\limsup_{k \to +\infty} E_k$. This follows from an adaptation of the same proof, see \S\ref{ssection:proof2}.

% As mentioned above, on $\{p=E\}$ (for $0 < E < E_c$) the orbits of the magnetic flow are all $2\pi \cdot T_E$-periodic. The orbit space $O_E := \{p=E\}/\sim$ (where $(x,\xi) \sim \Phi_t(x,\xi)$ for all $(x,\xi) \in \{p=E\}$ and $t \in \R$) are diffeomorphic to the surface $\Sigma$ itself. The measure $\mu$ (in Theorem \ref{theorem:constant2}), being invariant under $H_p^{\Omega}$, descends to a measure $\nu$ on $O_E$. In addition, as $\{p,\langle V \rangle\}^{\Omega} = 0$ (Poisson bracket), one has that $[H_p^{\Omega},H_{\langle V \rangle}^\Omega]=0$, and thus $H_{\langle V \rangle}^\Omega$ descends to a vector-field $Y$ on $O_E$. Theorem \ref{theorem:constant2} can then be reformulated in the following form: the measure $\nu$ on $O_E$ is invariant by the flow generated by $Y$.

As mentioned above, for $E \in [0, E_c]$, the orbit space $O_E$ is
diffeomorphic to the surface $\Sigma$ itself. The measure $\mu$ (in Theorem
\ref{theorem:constant2}), being invariant under $H_p^{\Omega}$, descends to a
measure $\nu$ on $O_E$. In addition, as $H_p^{\Omega}. \langle V \rangle = 0$, one has that $[H_p^{\Omega},H_{\langle V \rangle}^\Omega]=0$, and thus $H_{\langle V \rangle}^\Omega$ descends to a vector-field $Y$ on $O_E$. Theorem \ref{theorem:constant2} can then be reformulated in the following form: the measure $\nu$ on $O_E$ is invariant by the flow generated by $Y$.

\subsection{Literature}

\label{ssection:literature}

Semiclassical magnetic Laplacians on closed hyperbolic surfaces have been studied from
various perspectives by several authors, see the physics references \cite{Comtet-86,Iengo-Li-94} for the periodic regime, \cite{Guillemin-Uribe-89} for
the trace formula, and \cite{Kordyukov-Taimanov-22} for a more recent analysis of the trace formula near the critical energy—though the latter does not address semiclassical measures, as we do here. As previously mentioned, parts (ii) and (iii) of Theorem \ref{theorem:main} were already observed by Zelditch in \cite{Zelditch-92}.

Part (i) of Theorem \ref{theorem:main} and Theorem \ref{theorem:constant2} are also reminiscent of analogous results established on Zoll manifolds—that is, Riemannian manifolds on which all geodesics are closed—see \cite{ColinDeVerdiere-79, Uribe-Zelditch-93, Zelditch-97,Macia-08,Macia-09,Zelditch-15, Macia-Riviere-16,Arnaiz-Macia-22,Arnaiz-Macia-22-2}, among others. In particular, we use a now standard trick called Weinstein's averaging method \cite{Weinstein-77} to treat the perturbation by a small potential.

Part (iii) of Theorem \ref{theorem:main} bears similarity to Shnirelman's
result on Quantum Unique Ergodicity (QUE), see
\cite{Shnirelman-74-1,Shnirelman-74-2,Colindeverdiere-85,Zelditch-87}. This is closely related to the QUE
conjecture formulated by Rudnick and Sarnak \cite{Rudnick-Sarnak-94}, which
posits that on manifolds with negative curvature, the Liouville measure is the
unique semiclassical defect measure associated to high-frequency limits of
Laplace eigenfunctions. See also \cite{Dyatlov-22} for a review of recent developments.

%\cite{Colindeverdiere-85,Zelditch-87,Lindenstrauss-06,Anantharaman-08,Anantharaman-Nonnenmacher-07,Riviere-10,
%  Dyatlov-Jin-18,Dyatlov-Jin-Nonnenmacher-22, Bourgain-Dyatlov-18, Dyatlov-22}
%for related developments.

Under some different assumptions (horizontal Laplacians on flat principal bundles over Anosov Riemannian manifolds), a similar statement to (iii), Theorem \ref{theorem:main}, was established in \cite[Theorem 5.1.8]{Cekic-Lefeuvre-24} (see also \cite{Ma-Ma-23} for related developments).

Finally, Theorem \ref{theorem:horocyclic} should be compared with earlier results by Marklof and Rudnick \cite{Marklof-Rudnick-00} (see also the work by Rosenzweig \cite{Rosenzweig-06}) which appear to constitute the first instance of quantitative Quantum Unique Ergodicity. More recently, Morin and Rivière also obtained similar results for magnetic Laplacians on the torus \cite{Morin-Riviere-24}. The proof of Theorem \ref{theorem:horocyclic} makes use of a quantitative version of unique ergodicity due to Burger \cite{Burger-90}. 
%\subsection{Idea of proof} The proof of Theorem \ref{theorem:main} is very simple and is based on the following observation. Let $E < E_c$ be a subcritical energy.
%
%$\Pi_m$ is a Toeplitz operator; in our case, it is also a(n) (explicit) semiclassical Fourier Integral Operator and $\Delta \Pi_m \delta_u = \lambda_m \Pi_m \delta_u$. Furthermore, {\color{red} todo} \\

\subsection{Organization of the paper} In \S\ref{section:dynamical-geometric-background}, we recall elementary facts on the dynamics of the magnetic flow, and compute the first eigenvalues of the magnetic Laplacian in constant curvature. In \S\ref{section:microlocal}, we provide some background material on the twisted semiclassical quantization. In \S\ref{section:constant}, we construct eigenstates concentrating on periodic orbits of the magnetic flow in the low-energy regime. Finally, Theorems \ref{theorem:main}, \ref{theorem:horocyclic} and \ref{theorem:constant2} are proved in \S\ref{section:proofs}.\\

\noindent \textbf{Acknowledgement:} We thank Gabriel Rivière and Léo Morin for helpful discussions. This project was supported by the European Research Council (ERC) under the European Union’s Horizon 2020 research and innovation programme (Grant agreement no. 101162990 — ADG).

\section{Background}

\label{section:dynamical-geometric-background}

Throughout this section, we let $(\Sigma, g)$ denote a closed, connected, oriented hyperbolic surface—that is, a Riemannian surface with constant curvature $-1$. We establish the following results:

\begin{itemize} \item In \S\ref{ssection:geo}, we recall several technical results concerning the geometry and dynamics on the tangent bundle $T\Sigma$. In particular, we emphasize the unique ergodicity of the horocyclic flow, a key ingredient in the proof of Theorem~\ref{theorem:horocyclic}.
\item In \S\ref{ssection:magnetic-flow}, we introduce the magnetic flows on both the tangent and cotangent bundles, and show that they are equivalent under the identification induced by the metric. We also prove a technical result on the propagation of functions along the magnetic flow, which will be instrumental in the proof of Theorem~\ref{theorem:horocyclic}.
\item Finally, in \S\ref{ssection:landau}, we review standard properties of magnetic Laplacians. In particular, we explicitly compute the bottom of their spectrum—the so-called Landau levels. \end{itemize}
%
%{\coml{ Je préfèrerais que l'on commence par la partie que j'avais rédigée (qui
%  est maintenant la partie 2.2). J'avais fait l'effort de rédiger les résultats de base sur le flot
%  magnétique en partant de l'équation de Newton, et sans introduire trop de
%  géo diff, pour que cela reste accessible, et le tout de manière concise.} }

\subsection{Geometry and dynamics on the tangent bundle}

\label{ssection:geo}
Let $T\Sigma$ be the tangent bundle of $\Sigma$ and
\begin{equation}
\label{equation:footpoint}
\pi : T\Sigma \to \Sigma
\end{equation}
be the footpoint projection. We refer the reader to \cite[Chapter 1]{Paternain-99} or \cite[Chapter 13]{Lefeuvre-book} for background material on the geodesic flow.

\subsubsection{Structural equations} The geodesic flow $\varphi_t : T\Sigma \to T\Sigma$ is defined as
\[
\varphi_t(x,v) := (\gamma(t), \dot{\gamma}(t)),
\]
where $\gamma : \R \to \Sigma$ is the (unique) curve solving Newton's equation:
\begin{gather} \label{eq:newton_geodesic}
\nabla_{\dot{\gamma}(t)}\dot{\gamma}(t) = 0, \qquad \gamma(0)=x, \;  \dot{\gamma}(0)=v,
\end{gather}
and $\nabla$ is the Levi-Civita connection. We let
\[
X := \partial_t\varphi_t|_{t=0} \in C^\infty(T\Sigma, T(T\Sigma))
\]
be its generator.

As $\Sigma$ is oriented, there is a well-defined fiberwise rotation $R_\theta : T\Sigma \to T\Sigma$ by angle $\theta \in [0,2\pi)$ in the fibers of $T\Sigma$. Let $V$ be the generator of this rotation, namely
\[
V := \partial_\theta R_\theta|_{\theta = 0}.
\]
Let $V_\perp$ be the Euler vector field of $T\Sigma$, i.e. the generator of the flow
\begin{equation}
\label{equation:flot-vperp}
\varphi_t^{V_\perp}(x,v) = (x,e^tv).
\end{equation}
Finally, define
\[
X_\perp := [V,X].
\]
The vector fields $\{X,X_\perp,V,V_\perp\}$ form a basis of $T(T\Sigma)$ at any point of $T\Sigma$. The metric for which this is an orthonormal frame is called the \emph{Sasaki metric}.

We will need the following result:

\begin{lemma}
The above vector fields satisfy the commutation relations:
\begin{equation}
\label{equation:commutator-relations}
\begin{split}
[X,X_\perp] = - |v|^2 V, \qquad [X,V] =-X_\perp, \qquad [X_\perp,V] = X, \\
[V_\perp,X] = X, \qquad [V_\perp, X_\perp] = X_\perp, \qquad [V,V_\perp] = 0.
\end{split}
\end{equation}
\end{lemma}

We refer to \cite[Lemma 15.2.1]{Lefeuvre-book} for a proof. (In this reference, $X_\perp$ is denoted by $H$; the computations are done on the unit tangent bundle but the above relations follow easily by a scaling argument.) Notice that in our case, the sectional curvature is $\kappa = -1$.

% For later use, we record a technical remark:

% \begin{remark}
% Consider $w \in T_{(x,v)}(T\Sigma)$ and $z(t)=(x(t),v(t)) \in T\Sigma$ be a path such that $\dot{z}(0) = w$. Then 
% \begin{equation}
% \label{equation:connection}
% \mc{K}(w) := \nabla_{\dot{x}(t)}v(t),
% \end{equation}
% defines the \emph{connection map} $\mc{K} : T(T\Sigma) \to T\Sigma$. The tangent space $T(T\Sigma)$ splits as
% \[
% T(T\Sigma) = \HH \oplus \V,
% \]
% where $\HH = \ker \mc{K}$ is the horizontal space, $\V = \ker \dd \pi$ is the vertical space, and $\pi : T\Sigma \to \Sigma$ is the projection. The maps $\dd \pi : \HH \to T\Sigma$ and $\mc{K} : \V \to T\Sigma$ are isomorphisms.

% A quick computation reveals that
% \[
% X(x,v) = \dd\pi_{(x,v)}^{-1}(v) \in \HH, \qquad X_\perp(x,v) = \dd\pi_{(x,v)}^{-1}(j(x)v) \in \HH,
% \]
% where $j \in C^\infty(\Sigma,T\Sigma)$ is the almost-complex structure on $\Sigma$, that is the rotation by angle $+\pi/2$ induced by the conformal class of the metric $g$, and the orientation on $\Sigma$. In addition,
% \[
% V(x,v) = \mc{K}_{(x,v)}^{-1}(jv) \in \V, \qquad V_\perp(x,v) = \mc{K}_{(x,v)}^{-1}(v) \in \V,
% \]
% are the two vertical vector fields.
% \end{remark}

\subsubsection{Horocyclic flow} For $\lambda \geq 0$, define
\[
S^\lambda\Sigma := \{|v|_g=\lambda\}.
\]
We let $S\Sigma := S^{1}\Sigma$ be the unit tangent bundle. It is straightforward to verify that $X,X_\perp$ and $V$ preserve the layers $S^\lambda\Sigma$.

The geodesic flow $(\varphi_t)_{t \in \R}$ is Anosov on $S\Sigma$ in the sense that
\[
T(S \Sigma) = \R X \oplus E_s \oplus E_u,
\]
where 
\[
\begin{split}
|d\varphi_t(w)| \leq e^{-t} |w|, &\qquad \forall t \geq 0, w \in E_s \\
|d\varphi_{-t}(w)| \leq e^{-t} |w|, &\qquad \forall t \geq 0, w \in E_u.
\end{split}
\]
Here $|\bullet|$ denotes the norm induced by the \emph{Sasaki metric} on $T\Sigma$, which is defined as the metric for which $\{X,X_\perp,V,V_\perp\}$ is an orthonormal frame.

We define the (stable) horocyclic vector field by
\[
U_+ := X_\perp - V.
\]
Let $(h_t)_{t \in \R}$ be the (stable) \emph{horocyclic flow} generated by $U_+$ on $S\Sigma$. Recall that the Liouville measure on $S\Sigma$ is the unique volume form $\mu_{\mathrm{Liouv}}$ such that $\mu_{\mathrm{Liouv}}(X,X_\perp,V) = \mathrm{cst}$. The constant is normalized such that $\mu_{\mathrm{Liouv}}$ is a probability measure on $S\Sigma$. The Liouville measure is invariant by $X, X_\perp$ and $V$; it is thus also invariant by $U_+$.

\begin{theorem}
\label{lemma:horocycle}
The following holds:
\begin{enumerate}[label=\emph{(\roman*)}]
\item The flow $(h_t)_{t \in \R}$ is uniquely ergodic on $S\Sigma$ and $\mu_{\mathrm{Liouv}}$ is the only flow-invariant probability measure.
\item Let $0 < \theta < 1/2$ be such that $\theta(1-\theta) \leq \lambda_1(\Sigma)$, where $\lambda_1(\Sigma) > 0$ is the first non-zero eigenvalue of the Laplacian $\Delta_g$ on functions. Then there exists $C > 0$ such that for all $a \in C^{\infty}(S\Sigma)$, for all $T > 0$,
\[
\sup_{v \in S\Sigma} \left|\dfrac{1}{T}\int_0^T a(h_t(v))~ \dd t - \int_{S\Sigma} a(v) ~\dd\mu_{\mathrm{Liouv}}\right| \leq C T^{-\theta} \|a\|_{H^3(S\Sigma)},
\]
where $H^3(S\Sigma)$ denotes the $L^2$-based Sobolev norm of order $3$.
\end{enumerate}
\end{theorem}

Unique ergodicity in constant curvature was established by Furstenberg \cite{Furstenberg-73}, while the rate of convergence of ergodic averages was proved by Burger \cite{Burger-90}. A refinement of this result can be found in \cite{Flaminio-Forni-03}.

\subsection{Magnetic hyperbolic flow}

\label{ssection:magnetic-flow}

In this paragraph, we discuss the magnetic flow both on the tangent and the cotangent bundles.

\subsubsection{Magnetic flow on the tangent bundle}
\label{sssection:magnetic-tsigma}
A {\em magnetic geodesic}
is a curve $\ga $ of $\Sigma$  which satisfies the equation
\begin{equation} \label{eq:Newton_Lorentz}
\nabla_{\dot{\gamma}(t)} \dot{\gamma}(t) = - B j_{\gamma(t)}\dot{\ga} (t),
\qquad \gamma(0)=x,\; \dot{\gamma}(0)=v
\end{equation}
where $(x,v) \in T\Sigma$, $\nabla$ stands for the Levi-Civita covariant derivative, $j$ is the almost-complex structure and
$ B \in C^\infty(\Sigma)$ is the magnetic {\em intensity}. In the sequel we assume that $B$ is constant
and positive.  
The {\em magnetic flow} $(\Phi_t)_{t \in \R}$ is the autonomous flow of $T \Si$ such that for
any curve $\ga$ satisfying \eqref{eq:Newton_Lorentz},
\[
\Phi_t (x,v) := ( \ga ( t), \dot {\ga} ( t)), \qquad \forall t \in \R.
\]
\begin{lemma} The magnetic flow is generated by the vector field $F := X - B V$.
\end{lemma}

\begin{proof} We work in a chart $U$ of $\Si$ that we identify with an open
  subset of $\R^2$, so $TU = U \times \R^2$. Let $(x,v ) \in TU$. Let $c_1$,
  $c_2$ be a geodesic  and a magnetic geodesic respectively such that $ c_1
  ( 0 ) = x = c_2 (0)$, $\dot{c}_1 (0 ) =v = \dot{c}_2 ( 0)$. Then $X(x,v) = (
  \dot{c}_1 (0), \ddot{c}_1 (0))$ and $F(x,v) = (
  \dot{c}_2 (0), \ddot{c}_2 (0)) $, thus
  $$ F(x,v ) - X(x,v) = ( 0 , \ddot{c}_2 (0) -  \ddot{c}_1 (0)) = ( 0 ,
  ( \nabla_{\dot{c}_2} \dot{c}_2 ) (0) - ( \nabla_{\dot{c}_1} \dot{c}_1 ) (0)
  ) = ( 0 , -B j_{x} v ), $$
  where we have used that $\nabla_{\dot{c}_i} \dot{c}_i = \ddot{c}_i +
  \Gamma_{c_i} ( \dot{c}_i ) ( \dot{c}_i )$, $\Gamma$ being the connection
  one-form, and then Newton's equations \eqref{eq:newton_geodesic},
  \eqref{eq:Newton_Lorentz}. To conclude use that  $V  (x,v) = ( 0 , j_x v)$.      
\end{proof}

\begin{remark}The flow $(\Phi_t)_{t \in \R}$ is considered here on the tangent bundle $T\Sigma$, whereas it was initially defined on the cotangent bundle $T^*\Sigma$ in the introduction. In \S\ref{sssection:flow-cotangent}, we show that these two flows are equivalent under the musical isomorphism induced by the metric. To avoid unnecessary notation, unless specified explicitely, we do not distinguish between the flows on $T\Sigma$ and $T^*\Sigma$ in what follows.
\end{remark}

Recall that $(\varphi_t)_{t \in \R}$ is the geodesic flow (generated by $X$), $(R_t)_{t \in \R}$ is the $2\pi$-periodic rotation in the circle fibers of $S\Sigma$ (generated by $V$).

\begin{proposition} 
\label{proposition:dynamic}
For any $\lambda> 0$, the following holds:
  \begin{enumerate}[label=\emph{(\roman*)}]
    \item If $ \lambda< B$, $(\Phi_t)_{t \in \R}$ on $S^\lambda\Sigma$ is conjugate to $(R_{t/T_\lambda})_{t \in \R}$ on $S\Sigma$, where $T_\lambda = (B^2 - \lambda^2 )
      ^{-\frac{1}{2}}$;
      \item If $ \lambda = B$, then $(\Phi_t)_{t \in \R}$ on $S^\lambda\Sigma$ is conjugate to $(h_t)_{t \in \R}$ on $S\Sigma$; 
    \item If $ \lambda>B$, then $(\Phi_t)_{t \in \R}$ on $S^\lambda\Sigma$ is conjugate to $(\varphi_{t/T'_\lambda})_{t \in \R}$ on $S\Sigma$, where $ T_\lambda' = ( \lambda^2 -B^2 )
      ^{-\frac{1}{2}}$.
    \end{enumerate}
  \end{proposition}

\begin{proof}
For $\lambda > 0$, let $\Psi_\lambda : T\Sigma \to T\Sigma$ be the map defined by $\Psi_\lambda(x,v) := (x,\lambda v)$. Observe that $\Psi_\lambda^* X = \lambda X$ and $\Psi_\lambda^*V=V$. The flow of $F = X-BV$ on $S^\lambda \Sigma$ is thus conjugate to the flow of $\Psi_\lambda^* F = \lambda X -BV$ on $S\Sigma$.

Writing $S\Sigma = \Gamma\backslash\mathrm{PSL}(2,\R)$, the vector field $\lambda X - BV$ is represented by the following matrix element of $\mathfrak{sl}(2,\R)$:
\[
\lambda \begin{pmatrix} 1/2 & 0 \\ 0 & -1/2 \end{pmatrix} - B \begin{pmatrix} 0 & 1/2 \\ -1/2 & 0\end{pmatrix} = \begin{pmatrix} \lambda/2 & -B/2 \\ B/2 & -\lambda/2 \end{pmatrix}.
\]
For $\lambda = 1, B=0$, this is the generator of the geodesic flow; for $\lambda=0$, $B=1$, this is the generator of the rotation flow in the circle fibers, see \cite[Section 2.1]{Flaminio-Forni-03} for details. The eigenvalues of this matrix are $\pm\tfrac{1}{2}(\lambda^2-B^2)^{1/2}$ if $\lambda \geq B$ and $\pm \tfrac{i}{2}(B^2-\lambda^2)^{1/2}$ if $B \geq \lambda$. Notice that for $\lambda = B$, one finds the matrix
\[
\lambda/2 \begin{pmatrix} 1 & -1 \\ 1 & -1 \end{pmatrix}
\]
which is conjugate by an element of $\mathrm{PSL}(2,\R)$ to
\[
\begin{pmatrix} 0 & 1 \\ 0 & 0 \end{pmatrix},
\]
the generator of the horocyclic flow.

If two matrices in $\mathfrak{sl}(2,\R)$ are conjugate (by an element of $\mathrm{PSL}(2,\R)$), then their corresponding flows are conjugate too. Hence, if $B > \lambda$, $(\Phi_t)_{t \in \R}$ on $S^\lambda\Sigma$ is conjugate to $(R_{t(B^2-\lambda^2)^{1/2}})_{t \in \R} = (R_{t/T_\lambda})_{t \in \R}$ on $S\Sigma$. If $B=\lambda$, $(\Phi_t)_{t \in \R}$ on $S^\lambda\Sigma$ is conjugate to $(h_t)_{t \in \R}$ on $S\Sigma$. If $B < \lambda$, $(\Phi_t)_{t \in \R}$ on $S^\lambda\Sigma$ is conjugate to $(\varphi_{t(\lambda^2-B^2)^{1/2}})_{t \in \R} = (\varphi_{t/T'_\lambda})_{t \in \R}$ on $S\Sigma$.
\end{proof}

\subsubsection{Magnetic flow on the cotangent bundle} \label{sssection:flow-cotangent}

We now define the magnetic flow on the cotangent bundle, and show that it coincides with that on the tangent bundle up to the identification provided by the metric. Let
\[
\flat : T\Sigma \to T^*\Sigma, \qquad \flat(x,v) := (x, g_x(v,\bullet))
\]
be the musical isomorphism. Let $A$ be the Liouville
form on $T^*\Sigma$ defined by
\begin{equation}
\label{equation:liouville}
A_{(x, \xi) }  := \xi ( \dd_{x, \xi} \pi (\bullet)),
\end{equation}
where $\pi : T^*\Sigma \to \Sigma$ denotes the projection. Consider the symplectic form
\[
\Om \in C^\infty(T^* \Si,\Lambda^2 T^*(T^*\Sigma)), \qquad \Om := dA    +  B~ \pi^*\op{vol},
\]
where $\op{\vol} \in C^\infty(\Si,\Lambda^2 T^*\Sigma)$ the Riemannian volume. Let $p(x,\xi) := \tfrac{1}{2} | \xi|^2$, and define the flow $(\Phi_t^{T^*\Sigma})_{t \in \R}$ as the Hamiltonian flow of the function $p \in C^\infty(T^*\Sigma)$ with respect to the symplectic form $\Omega$. Namely $(\Phi_t^{T^*\Sigma})_{t \in \R}$ is generated by $H^{\Omega_p}$ such that
\[
dp = \Omega(\bullet, H^\Omega_p).
\]
To avoid confusion, let us denote temporarily $(\Phi_t^{T\Sigma})_{t \in \R}$ the magnetic flow defined on $T\Sigma$ in \S\ref{sssection:magnetic-tsigma}.

\begin{lemma}
The map $\flat$ intertwines $(\Phi_t^{T\Sigma})_{t \in \R}$ and $(\Phi_t^{T^*\Sigma})_{t \in \R}$, that is
\[
\Phi_t^{T^*\Sigma} \circ \flat = \flat \circ \Phi_t^{T\Sigma}, \qquad \forall t \in \R.
\]
\end{lemma}

\begin{proof}
Define the $1$-form $\alpha := \flat^*A$ on $T\Sigma$. It satisfies
\[
\alpha_{(x,v)} = g_x(\dd_{x,v}\pi(\bullet),v),
\]
see \cite[Lemma 1.37, item 1]{Paternain-99}. Hence
\[
\flat^*\Omega = d\alpha + B~ \pi^*\vol,
\]
where $\pi : T\Sigma \to \Sigma$ is the footpoint projection. Define $h(x,v) := |v|^2_g/2$. Let $H_h^{\flat^*\Omega}$ be the Hamiltonian vector field of $h$, computed with respect to $\flat^*\Omega$ on $T\Sigma$, namely
\[
dh = \flat^*\Omega(\bullet, H_h^{\flat^*\Omega}).
\]
It is immediate that $\flat^* H_p^{\Omega} = H_h^{\flat^*\Omega}$ so the claim boils down to showing that
\[
H_h^{\flat^*\Omega} = X - B V.
\]
For that, we compute:
\begin{equation}
\label{equation:mouch}
\flat^*\Omega(\bullet,X-BV) = \dd\alpha(\bullet, X) - B \dd \alpha(\bullet, V) + B \pi^*\vol(\bullet, X).
\end{equation}
A quick computation reveals that $\dd \alpha(\bullet, X) = dh$. Indeed, $\dd \alpha(Y,X) = 0 = dh(Y)$ for $Y = X,X_\perp,V$ and $\dd \alpha(V_\perp,X) = |v|^2 = dh(V_\perp)$ using the $2$-homogeneity of $h$ and that $V_\perp$ is the Euler vector field.

In addition, using \cite[Proposition 1.24]{Paternain-99} for the formula for $\dd \alpha$, we find:
\[
\dd \alpha(\bullet,V) = - g_x(j(x)v, \dd \pi(\bullet)) = \pi^*\vol(\bullet,X),
\]
so the last two terms in \eqref{equation:mouch} cancel out. This proves that $\flat^*\Omega(\bullet,X-BV) = dh$, and thus $H_h^{\flat^*\Omega} = X - B V$.
\end{proof}

In what follows, we will often implicitly identify the magnetic flow on $T^*\Sigma$ and $T\Sigma$. We also drop the index $T\Sigma$ or $T^*\Sigma$ on the magnetic flow $(\Phi_t)_{t \in \R}$.

%\begin{remark}{\color{red}TL: j'enlèverais cette remarque.}Notice that $p$ depends only on $g$ and $\Om$ on $B$. Using
%temporarily the notations $p_g$, $\Om_B$ and $H_{g,B}$ for the Hamiltonian
%vector field generating $(\Phi_t)_{t \in \R}$, we notice that for any positive constants $c$ and $\lambda$, the
%map $\Phi ( x, \xi ) = ( x, \lambda \xi)$ satisfies $\Phi^* \Om_{\la B} =
%\lambda \Om_B$ and $\Phi ^* p_{cg } = \la^{2} c^{-1} p_g$, so that $\Phi ^{-1}$
%sends $H_{cg , \la B}$ to $\la c^{-1} H_{g, B}$. As a consequence, there is no loss of
%generality in normalizing the
%curvature and the magnetic intensity. The sectional curvature is already normalized to $-1$ (hyperbolic metric). From \S\ref{section:constant} on, we will also further assume that $B=1/|\chi(\Sigma)|$. 
%\end{remark}

%The vector fields $V$ and $V_\perp$ are $0$-homogeneous (that is $\varphi_t^{V_\perp}(V) = V$ and the same holds for $V_\perp$) but $X$ and $X_\perp$ are $1$-homogeneous (that is $\varphi_t^{V_\perp}(X) = e^t X$ and the same holds for $X_\perp$. In the following, it will be convient to consider the $0$-homogeneous vector fields $X^0 := X/|v|$ and $X_\perp^0 := X_\perp/|v|$. They satisfy the commutation relations:
%\begin{equation}
%\begin{split}
%[X^0,X_\perp^0] = \kappa V, \qquad [X^0,V] =-X_\perp^0, \qquad [X_\perp^0,V] = X^0 \\
%[V_\perp,X^0] = 0, \qquad [V_\perp, X_\perp] = 0, \qquad [V,V_\perp] = 0,
%\end{split}
%\end{equation}

\subsubsection{Propagation estimate}
The following lemma will be crucial in the proof of Theorem \ref{theorem:horocyclic}; it computes effectively the Lyapunov exponents of the magnetic flow on the energy layers $\{p \leq E\}$. We write $x_+ := \max(x,0)$ for $x \in \R$. To uniformize notation, we also let $p(x,v) := \tfrac{1}{2}|v|^2$ on $T\Sigma$ (instead of $h$ as in \S\ref{sssection:flow-cotangent}).

\begin{lemma}[Propagation estimate]
\label{lemma:wow}
Let $n \geq 0$. There exist a constant $C_n > 0$ and an integer $m_n \geq 0$ (with $m_0 = 0$) such that for all $E \in [0,10E_c]$, for all $f \in C^\infty_{\comp}(T\Sigma)$ with $\supp(f) \subset \{p \leq E\}$, for all $t \geq 0$:
\begin{equation}
\label{equation:key-bound}
\|f \circ \Phi_t\|_{C^n(T\Sigma)} \leq C_n \langle t \rangle^{m_n} e^{\sqrt{2}(E-E_c)_+^{1/2}n t}\|f\|_{C^n(T\Sigma)}.
\end{equation}
\end{lemma}

The proof actually gives
\[
m_n = 3n + n(n+1)/2.
\]
The important point is that $C_n$ is independent of the maximal energy layer $\{p=E\}$ supporting $f$. We emphasize that Lemma \ref{lemma:wow} is stated on the tangent bundle $T\Sigma$ but it also holds for the horocyclic flow on $T^*\Sigma$ (see the discussion at the end of \S\ref{ssection:magnetic-flow}).

\begin{proof}
We define inductively the $C^n$ norm for $n \geq 0$ by
\begin{equation}
\label{equation:definition-cn}
\|f\|_{C^{n+1}(T\Sigma)} := \|f\|_{C^n(T\Sigma)} + \sum_{Z = X,X_\perp,V,V_\perp} \|Zf\|_{C^n(T\Sigma)}.
\end{equation}
We prove the claim by iteration on $n$. Fix $v \in T\Sigma$ and $Z \in T_{v}(T\Sigma)$ (we drop the footpoint $x$ in the notation). We write for $t \in \R$:
\[
d\Phi_t(Z) = a(t,v) X(\Phi_t(v)) + b(t,v) X_\perp(\Phi_t(v)) + c(t,v) V(\Phi_t(v)) + d(t,v) V_\perp(\Phi_t(v)).
\]
Precomposing with $(d\Phi_t)^{-1} $, differentiating with respect to $t \in \R$ and using \eqref{equation:commutator-relations} to compute the Lie brackets appearing, one verifies that (we drop the $v$ in the notation):
\begin{align}
\label{equation:at} \dot{a}(t)- B b(t) - d(t) &= 0, \\
\label{equation:bt} \ddot{b}(t)+(B^2-|v|^2)b(t)-Bd(t)&=0, \\
\label{equation:ct} \dot{c}(t)-|v|^2b(t)&=0, \\
\label{equation:dt} \dot{d}(t) &= 0.
\end{align}
(The equation obtained for $\dot{b}(t)$ is $\dot{b}(t) - Ba(t)-c(t) = 0$. Differentiating once again with respect to $t \in \R$ and using \eqref{equation:at} and \eqref{equation:ct}, one obtains \eqref{equation:bt}.)

By \eqref{equation:dt}, $d(t)=d(0)$. Inserting this in \eqref{equation:bt}, we find that for $|v| < B$,
\begin{equation}
\label{equation:bt-complete1}
\begin{split}
b(t,v) &= \dfrac{Bd(0)}{B^2-|v|^2}\left(1-\cos((B^2-|v|^2)^{1/2}t) \right) + b(0) \cos((B^2-|v|^2)^{1/2}t) \\
&\hspace{6cm} + \dfrac{Ba(0)+c(0)}{(B^2-|v|^2)^{1/2}} \sin((B^2-|v|^2)^{1/2}t).
\end{split}
\end{equation}
For $|v|=B$:
\begin{equation}
\label{equation:bt-complete2}
b(t,v) = \dfrac{Bd(0)}{2} t^2 + b(0) + (Ba(0)+c(0))t.
\end{equation}
(Notice that one recovers \eqref{equation:bt-complete2} from \eqref{equation:bt-complete1} by letting $|v| \to B$.) For $|v| > B$:
\begin{equation}
\label{equation:bt-complete3}
\begin{split}
b(t,v) &= \dfrac{Bd(0)}{B^2-|v|^2}\left(1-\cosh((|v|^2-B^2)^{1/2}t) \right) + b(0) \cosh((|v|^2-B^2)^{1/2}t) \\
&\hspace{6cm} + \dfrac{Ba(0)+c(0)}{(|v|^2-B^2)^{1/2}} \sinh((|v|^2-B^2)^{1/2}t).
\end{split}
\end{equation}
Using \eqref{equation:at} and \eqref{equation:ct}, we then find:
\begin{align}
\label{equation:at-complete}
a(t,v) &= a(0) + td(0) + B\int_0^t b(s,v) \dd s, \\
\label{equation:ct-complete}
c(t,v) &= c(0) + |v|^2 \int_0^t b(s,v) \dd s.
\end{align}
It follows from \eqref{equation:bt-complete1}, \eqref{equation:bt-complete2} and \eqref{equation:bt-complete3} that for all $n \geq 0$, there exists $C_n > 0$ such that:
\begin{equation}
\label{equation:bt-sharp}
\|b(t,\bullet)\|_{C^n(\{p \leq E\})} \leq C_n \langle t \rangle^{2+n} e^{\sqrt{2}(E-E_c)_+^{1/2}t}.
\end{equation}
Indeed, there exists a constant $C > 0$ such that for all $\omega \in (0,1)$, for all $t \geq 0$,
\[
\max\left(\omega^{-2}(1-\cosh(\omega t)), \omega^{-1}\sinh(\omega t)\right) \leq C \langle t\rangle^2 e^{\omega t},
\]
and the same holds by replacing respectively $\sinh$ and $\cosh$ by $\sin$ and $\cos$. Inserting the previous estimate in \eqref{equation:bt-complete1} and \eqref{equation:bt-complete3} proves \eqref{equation:bt-sharp} in the case $n=0$. Taking derivatives, and iterating the same argument, one obtains \eqref{equation:bt-sharp}. Then, inserting \eqref{equation:bt-sharp} in \eqref{equation:at-complete} and \eqref{equation:ct-complete}, we find that:
\begin{equation}
\label{equation:at-ct-sharp}
\|a(t,\bullet)\|_{C^n(T\Sigma)}, \|c(t,\bullet)\|_{C^n(T\Sigma)} \leq C_n \langle t\rangle^{3+n} e^{\sqrt{2}(E-E_c)_+^{1/2}t}.
\end{equation}

We now go back to \eqref{equation:definition-cn}. Our aim is to estimate:
\[
\|f\circ\Phi_t\|_{C^{n+1}(T\Sigma)} = \|f\circ\Phi_t\|_{C^n(T\Sigma)} + \sum_{Z = X,X_\perp,V,V_\perp} \|df \circ \Phi_t (d\Phi_t(Z))\|_{C^n(T\Sigma)},
\]
in order to prove \eqref{equation:key-bound} iteratively on $n \geq 0$. For $n = 0$, the claim is immediate with $m_0 = 0$ and $C_0 = 1$. Suppose that it holds for $n \geq 0$. Notice that
\[
\begin{split}
\|df \circ \Phi_t (d\Phi_t(Z))\|_{C^n} & = \|df \circ \Phi_t (a(t,\bullet) X + b(t,\bullet) X_\perp + c(t,\bullet) V + d(t,\bullet) V_\perp)\|_{C^n} \\
& \lesssim \|a(t,\bullet)\|_{C^n} \|(Xf)\circ\Phi_t\|_{C^n} + \|b(t,\bullet)\|_{C^n} \|(X_\perp f)\circ\Phi_t\|_{C^n}  \\
& \qquad + \|c(t,\bullet)\|_{C^n} \|(Vf)\circ\Phi_t\|_{C^n}  + \|d(t,\bullet)\|_{C^n} \|(V_\perp f)\circ\Phi_t\|_{C^n} \\
& \lesssim \langle t \rangle^{3+n} e^{\sqrt{2}(E-E_c)_+^{1/2}t} \langle t \rangle^{m_n} e^{\sqrt{2}(E-E_c)_+^{1/2}nt} \max_{Z=X,X_\perp,V,V_\perp}\|Z f\|_{C^n} \\
& \leq C_{n+1}\langle t \rangle^{m_{n+1}} e^{\sqrt{2}(E-E_c)_+^{1/2}(n+1)t} \|f\|_{C^{n+1}},
\end{split}
\]
for some uniform constants $C_{n+1} > 0$, where we used in the second line that $\|\bullet\|_{C^n(T\Sigma)}$ is an algebra norm, and in the third line the estimates \eqref{equation:bt-sharp} and \eqref{equation:at-ct-sharp} together with the inductive assumption \eqref{equation:key-bound}, and the fact that $f$ has support in the energy layers $\{p \leq E\}$. This proves the claim by setting $m_{n+1} := m_n + 3 + n$.
\end{proof}

\subsection{Magnetic hyperbolic Laplacian} \label{ssection:landau} By the Gauss-Bonnet formula, the metric being
hyperbolic, the volume is 
\begin{equation}
\int_{\Si} \vol = 4 \pi (\textsl{g} -1 ),
\end{equation}
where $\textsl{g}\geqslant
2 $ is the genus of $\Si$. Let $L \to \Sigma$ be a Hermitian line bundle with
a Hermitian connection
$\nabla$ whose curvature is $-i B \op{vol}
$, $B > 0$ being a positive constant. Since the Chern class of $L$ is $ (2
\pi)^{-1} B \cdot [ \op{vol}]$, the
degree of $L$ is 
\begin{equation}
\label{equation:deg}
\mathrm{deg}(L) =  (2 \pi)^{-1} B \int_{\Si} \vol = 2 B (\textsl{g}-1) \in \Z.
\end{equation}
Thus, the pair $( L ,
\nabla)$ exists if and only if $2 B (\textsl{g}-1)$ is an integer. When it exists, it
is unique up to tensor product by a flat Hermitian line bundle. The
space of flat Hermitian line bundles up to equivalence is in one-to-one
correspondance trough the holonomy representation with the character space
$\op{Mor} (\pi_1 ( \Si ) ,  \op{U}(1) ) \simeq \op{U} (1) ^{2\textsl{g}}$. An example of a pair $(L, \nabla)$ is provided by the canonical bundle $K := T^{1,0} \Si$,
endowed with its Chern connection, whose curvature is $- i \vol$, so in this
case $B=1$.

Let $\Delta := \frac{1}{2} \nabla^* \nabla $ be the Laplacian acting on $C^\infty
( \Si , L)$. Since $\Si$ is compact, $\Delta$ has a discrete spectrum and each
eigenvalue has finite multiplicity. We denote by $\la_0 \leqslant \la_1
\leqslant \ldots $ the eigenvalues of $\Delta$. The following result was
proved in \cite{Iengo-Li-94} (see also \cite{Tejero-06, Landau2}):

\begin{proposition} \label{theo:Landau_Level}
The following holds:
\begin{enumerate}[label=\emph{(\roman*)}]
\item For any pair $( L , \nabla)$ with curvature $F_\nabla=-iB
  \op{vol}$ with $B>0$, the first $N :=  \lfloor B \rfloor $ eigenvalues  are
  given by 
\begin{gather} \label{eq:valeur_propre}
  \la_{m } = B ( \tfrac{1}{2} + m ) - \dfrac{ m ( m+1 ) } {2},\qquad  0
  \leqslant m < N,
\end{gather}
and the multiplicity of $\la_m$ is $2 (\textsl{g}-1)(B-\frac{1}{2} - m)$ when $m \leqslant N- 2$.

\item If $(L,\nabla) = (K^r,\nabla^{\mathrm{Chern}})$ with $ r \geq 1$, then \eqref{eq:valeur_propre} holds for
$m \leqslant  r$ and the remaining part of the spectrum is given by 
 $\la_{r+n} = \la_r + \frac{1}{2} \mu_n$ where $\{ \mu_n, \; n \geq 0 \}$ is the
  spectrum of the Laplace-Beltrami operator of $(\Sigma, g )$.
  \end{enumerate}
\end{proposition}

We call the eigenvalues $\la_m$ given by \eqref{eq:valeur_propre} and their eigenspaces the Landau levels. 
These quantum levels correspond to the classical energy
levels $S^\lambda \Si$ with $\lambda <B$ on which the magnetic flow is periodic. A first
clue of this fact is that \eqref{eq:valeur_propre} applied with $m = B$ gives $\frac{1}{2} B^2$ which is exactly the classical
energy $\frac{1}{2} \lambda^2$ at the critical value $\lambda =B$.
Moreover, as was noticed by Comtet \cite{Comtet-86}, \eqref{eq:valeur_propre} can be interpreted as a
Bohr-Sommerfeld condition.

On a different perspective, a common characteristic
of the magnetic flow below the critical energy  and the Landau levels is
that they do not depend on the choice of hyperbolic metric $g$. The part of the spectrum above the
critical energy has a different nature (as seen with $L = K^r$ for instance).

\begin{proof}[Sketch of proof] This follows from the Riemann-Roch theorem and the
  following two identities
  \begin{gather} \label{eq:bochner_bosonic}
    \Delta = \Box_L + \tfrac{1}{2} B, \qquad \con{\partial}_L
    \con{\partial}_L^* = \Box_{L \otimes K^{-1}} +B -1 .
  \end{gather}
  Here, $L$ is equipped with its holomorphic structure whose Chern connection
  is $\nabla$,  $\con{\partial}_L: C^{\infty} ( L )
  \rightarrow C^{\infty} ( L \otimes K^{-1})$ is the $d$-bar
  operator with the identification  $\con{ K} = K^{-1}$ induced by
  the metric, and $\Box_L = \con{\partial}_L^*
    \con{\partial}_L$. The first identity of \eqref{eq:bochner_bosonic} is a
    classical Bochner identity, the second one is proved in \cite[Theorem 7.1]{Landau1}. 

    For any $m \in \N$, let $\con{\partial}_m = \con{\partial}_{L \otimes
      K^{-m}}$ and $\Box_m = \Box_{L \otimes K^{-m}}$. The curvature of $L
    \otimes K^{-m}$ being $\frac{1}{i} (B -m ) \op{vol}$, the second identity
    gives $$\con{\partial}_m \Box_m = (
    \Box_{m+1} + (B - m) - 1 ) \con{\partial}_m.$$
    Introduce the operator $\Box_{m}^{-1}$
     of $C^\infty( L \otimes K^{-m})$ equal to $0$ on $\ker \Box_m$
    and inverting $\Box_m$ on the orthogonal of $\ker \Box_m$. Then
    $\Box_m^{-1} \con{\partial}_m^*$ is the inverse of $\con{\partial}_m$ on
    $(\ker \Box_m)^{\perp}$, so by the previous equation
\begin{gather*}
    \Box_m = \Box_m^{-1} \con{\partial}_m^* \bigl( \Box_{m+1} + B - ( m+1)
    \bigr) \con{\partial}_m. 
  \end{gather*}
  As a consequence, the spectra satisfy
\begin{gather}  \label{eq:rec}
  \op{Sp} ( \Box_m ) \setminus \{ 0 \} = \op{Sp} (
    \Box_{m+1} ) + B - ( m +1 ). 
\end{gather}
  Moreover, since the degree of $L \otimes K^{-m}$ is $(B- m) 2 (\textsl{g} -1 )$, by
the Riemann-Roch theorem, $h^0 ( L \otimes K^{-m}) := \op{dim} ( \ker \Box_m )$ is
positive when $B - m \geqslant 1$, and equal to $(B-m) 2 (\textsl{g}-1) + 1 - \textsl{g}$ when
$B- m \geqslant 2 $.

We deduce the first part of the Theorem as follows: by the first identity of
\eqref{eq:bochner_bosonic}, $\Delta = \Box_0 + \frac{1}{2} B$ and if $B
\geqslant 1$, $\la_0 = \tfrac{1}{2} B$ and its multiplicity is $h^0 (
L)$. If $B\geqslant 2$, then $H^{0} (\Sigma, L \otimes K^{-1} ) >0$, so
$\la_1 =\frac{1}{2}B  + B -1$ by \eqref{eq:rec} and its multiplicity is $H^{0}
(\Sigma, L \otimes K^{-1} )$. 
We can repeat this until we can no longer apply the Riemann-Roch Theorem.

When $L = K^r$, after $r$ iterations, it comes that  $\op{Sp}(\Delta) \setminus \{ \la_0, \ldots , \la_{r-1} \}$
is equal to the spectrum of $\Box_r + \frac{1}{2} r^2$. By Bochner identity,
$\Box_r$ is the Laplace-Beltrami operator.
\end{proof}

%\subsubsection{Low-energy projection}
%
%Recall that the critical energy is $E_c := \tfrac{1}{2}B^2$. Let $\Pi_{m}$ be the orthogonal projector onto $\ker(\Delta-\lambda_m)$, and set
%\begin{equation}
%\label{equation:a}
%\mathbf{A} := \sum_{\lambda_m \leq E_c} m \Pi_m.
%\end{equation}
%Observe that by construction $\op{Sp}(\mathbf{A}) \subset \Z_{\geq 0}$ and thus
%\begin{equation}
%\label{equation:cool}
%e^{2i\pi \mathbf{A}} = \mathbf{1}.
%\end{equation}
%Furthermore, $\mathbf{A} : C^\infty(\Sigma,L) \to C^\infty(\Sigma,L)$ satisfies the identity
%\begin{equation}
%\label{equation:relation}
%\Delta = B(\tfrac{1}{2}+\mathbf{A})-\dfrac{\mathbf{A}(\mathbf{A}+1)}{2},
%\end{equation}
%on the space $\mc{I} := \oplus_{\lambda_m \leq E_c} \ker(\Delta-\lambda_m)$. Equivalently, 
%\begin{equation}
%\label{equation:relation2}
%\mathbf{A}= \tfrac{1}{2}B-\tfrac{1}{4} - \sqrt{\tfrac{1}{4}B^2+\tfrac{1}{4}B + \tfrac{1}{8}-\Delta},
%\end{equation}
%on $\mc{I}$. {\color{red} cette dernière formule est fausse, j'ai réécris ce
%  paragraphe plus loin, enlever la section 2.3.2} 

\section{Twisted semiclassical calculus}

\label{section:microlocal}

We now introduce the semiclassical limit we will be working with. In this section, $\Sigma$ need just be a closed manifold. 
Let $L \to \Sigma$ be a Hermitian line bundle with a Hermitian connection $\nabla$.  We do not
make any assumption on the curvature $F_{\nabla}$. This section is organized as follows:
\begin{itemize}
\item In \S\ref{ssection:uniform-pseudo}, we introduce the twisted semiclassical pseudodifferential calculus;
\item In \S\ref{ssection:defect-measures}, we discuss semiclassical defect measures associated to quasimodes in this calculus.
\end{itemize}

\subsection{Definition and main properties}

\label{ssection:uniform-pseudo}

 For $m \in
\R$, denote by $\Psi^m_{\mathrm{sc}}(\Sigma)$ the space of standard
($h$-dependent) semiclassical pseudodifferential operators of degree $m \in
\R$ (see \cite[Chapter 14]{Zworski-12} for an introduction). We emphasize that
$A \in \Psi^m_{\mathrm{sc}}(\Sigma)$ is a \emph{family} of operators $A =
(A_h)_{h > 0}$, where $A_h = \Op_h(a_h) + \mc{O}(h^\infty)$, $\Op_h$ is an
arbitrary semiclassical quantization on $\Sigma$, and $a_h \in
S^m_h(T^*\Sigma)$ satisfies uniform symbolic estimates in $h > 0$ (see \cite[Chapter 3]{Lefeuvre-book} for a definition of symbol classes). In addition, the family $A$ might not be defined for all values of $h > 0$, but only on a subset.

\begin{definition}[Twisted semiclassical quantization]
A family of operators $\mathbf{A} := (\mathbf{A}_k)_{k \geq 0}$ such that
\[
\mathbf{A}_{k} : C^\infty(\Sigma, L^{k}) \to C^\infty(\Sigma, L^{k}),
\]
belongs to the space of \emph{twisted} semiclassical operators $\Psi^m_{\mathrm{tsc}}(\Sigma)$ of degree $m \in \R$ if the following holds: for all contractible open subset $U \subset \Sigma$, for all $\chi,\chi' \in C^\infty_{{\comp}}(U)$, for all $s \in C^\infty(U,L)$ such that $|s|=1$ fiberwise, there exists a family of semiclassical operators $A := (A_h) \in \Psi^m_{\mathrm{sc}}(\Sigma)$ such that for all $f \in C^\infty_{{\comp}}(U)$:
\begin{equation}
\label{equation:def}
s^{-k}\mathbf{A}_k(f s^{k}) = A_{1/k} f.
\end{equation}
\end{definition}

This quantization was introduced by the first author \cite{Charles-00}. We
also refer to \cite[Section 3.2]{Cekic-Lefeuvre-24} and \cite[Section 2]{Charles--23} for a general introduction to this quantization procedure. We now recall some of its properties. 

Define on the open set $U$ the real-valued $1$-form $\beta \in C^\infty(U,T^*U)$ by $\nabla s = -i \beta \otimes s$ (hence $\nabla s^{k} = -i k \beta \otimes s^{k}$). We also introduce the symplectic $2$-form
\begin{equation}
\label{equation:omega}
\Omega := dA + i \pi^* F_\nabla,
\end{equation}
where $dA$ is the Liouville $2$-form on $T^*\Sigma$ (see \eqref{equation:liouville}) and $\pi : T^*\Sigma \to \Sigma$ is the projection. (For the magnetic Laplacian on surfaces, we will have $\Omega = dA + B \pi^*\vol$.) It can be easily checked that \eqref{equation:omega} defines a non-degenerate (closed) $2$-form. Given $p \in C^\infty(T^*\Sigma)$, the Hamiltonian vector field of $p$ associated to $\Omega$ is denoted by $H_p^\Omega$. It is defined through the relation
\[
dp = \Omega(\bullet, H_p^\Omega).
\]
Finally, the Poisson bracket of two observables $p,q \in C^\infty(T^*M)$ computed with respect to $\Omega$ is written as $\{p,q\}^\Omega =: H^\Omega_p q$. We sum up the main properties of this calculus (below, $S^m_k(T^*\Sigma)$ denotes the space of symbols of order $m$ which may depend on $k \geq 0$; however, the symbolic estimates are required to be uniform in $k$):

\begin{proposition}
The following properties hold:

\begin{enumerate}[label=\emph{(\roman*)}]
\item \emph{\textbf{Algebra property.}} The space
\[
\Psi^\bullet_{\mathrm{tsc}}(\Sigma) := \cup_{m \in \R} \Psi^m_{\mathrm{tsc}}(\Sigma)
\]
is a graded algebra. Namely, for all $\mathbf{A}\in\Psi^m_{\mathrm{tsc}}(\Sigma)$, $\mathbf{B}\in\Psi^{m'}_{\mathrm{tsc}}(\Sigma)$, $\mathbf{A} \circ \mathbf{B} \in \Psi^m_{\mathrm{tsc}}(\Sigma)$.

\item \emph{\textbf{Principal symbol.}} For $\mathbf{A} \in \Psi^m_{\mathrm{tsc}}(\Sigma)$, there is a well-defined principal symbol given by
\[
\sigma_{\mathbf{A}}(x,\xi ; k) = a(x,\xi + \beta(x)) \in
S^m_{k}(T^*\Sigma)/k^{-1}S^{m-1}_{k}(T^*\Sigma),
\]
where $a$ is the principal symbol of $A_{1/k}$.

\item \emph{\textbf{Commutator.}} For $\mathbf{A} \in \Psi^m_{\mathrm{tsc}}(\Sigma)$, $\mathbf{B} \in \Psi^{m'}_{\mathrm{tsc}}(\Sigma)$, $[\mathbf{A},\mathbf{B}] \in k^{-1}\Psi^{m+m'-1}_{\mathrm{tsc}}(\Sigma)$ with principal symbol
\[
\sigma_{k[\mathbf{A},\mathbf{B}]} = -i \{\sigma_{\mathbf{A}},\sigma_{\mathbf{B}}\}^\Omega.
\]

\item \emph{\textbf{Calderon-Vaillancourt.}} Let $\mathbf{A} \in \Psi^0_{\mathrm{tsc}}(\Sigma)$. Then $\mathbf{A}_k : L^2(\Sigma,L^k) \to L^2(\Sigma,L^k)$ is bounded for all $k \geq 0$ with norm
\[
\|\mathbf{A}_k\|_{L^2 \to L^2} \leq \|\sigma_{\mathbf{A}}(\bullet;k)\|_{L^\infty(T^*\Sigma)} + \mc{O}(k^{-1}).
\]

\end{enumerate}

\end{proposition}

By construction, the operator $\Delta_k := \tfrac{1}{2} (\nabla^k)^*\nabla^k \in \Psi^2_{\mathrm{tsc}}(\Sigma)$ belongs to this calculus and has principal symbol $p(x,\xi) := \sigma_{\Delta_k}(x,\xi) = \tfrac{1}{2}|\xi|^2_g$.

Conversely, one defines a quantization map $\Op : S^m_{k}(T^*\Sigma) \to \Psi^m_{\mathrm{tsc}}(\Sigma)$ as follows. Consider a smooth function $\chi \in C^\infty(\Sigma \times \Sigma)$, equal to $1$ near the diagonal $\Delta \subset \Sigma \times \Sigma$ and supported in $\{d(x,y) < \iota(g)/2024\}$, where $\iota(g)$ is the injectivity radius of the metric $g$. For $k \geq 0$, and $x,y \in \mathrm{supp}(\chi)$, let $\tau_{x \to y} : L^k_x \to L^k_y$ denote the parallel transport with respect to $\nabla$ along the unique $g$-geodesic joining $x$ to $y$. We then set:
\[
\Op_{k}(a_{k}) f(x) := \dfrac{1}{2\pi} \int_{\Sigma} \int_{T^*_x\Sigma} e^{-ik\xi(\exp^{-1}_x(y))} a_{k}(x,\xi) \tau_{y \to x}(f(y)) \chi(x,y) \dd\xi \dd y,
\]
where $\dd y$ is the Riemannian volume, and $\dd\xi$ the induced volume in the fibers of $T^*\Sigma$. It is a standard calculation to verify that
\[
\sigma_{\Op_{k}(a_{k})} = [a_{k}] \in S^m_{k}(T^*\Sigma)/k^{-1}S^{m-1}_{k}(T^*\Sigma).
\]

An operator $\mathbf{A} \in \Psi^\bullet_{\mathrm{tsc}}(\Sigma)$ is \emph{microlocally compactly supported} if $\mathbf{A} = \Op_k(a_k) + \mc{O}(k^{-\infty})$ where $a_k \in C^\infty_{\comp}(T^*\Sigma)$ has (uniformly in $k$) compact support on $T^*\Sigma$, and the remainder term $\mc{O}(k^{-\infty})$ denotes an operator with smooth Schwartz kernel on $\Sigma \times \Sigma$, all of whose derivatives are $\mc{O}(k^{-\infty})$ in $L^\infty$-norm. We denote by $\Psi^{\comp}_{\mathrm{tsc}}(\Sigma)$ the set of microlocally compactly supported operators.

Finally, we shall use the radial compactification $\overline{T^*\Sigma} :=
T^*\Sigma \sqcup \partial_\infty T^*\Sigma$ of cotangent space (obtained by
adding a sphere at infinity to each fiber of $T^*\Sigma$), see \cite[Appendix
E.1.3]{Dyatlov-Zworski-book}. As in standard semiclassical theory, an operator
$\mathbf{A} \in \Psi^m_{\mathrm{tsc}}(\Sigma)$ is \emph{elliptic} at $(x_0,\xi_0)
\in \overline{T^*\Sigma}$ if there exists a constant $c > 0$ such that for all
$k \geq 0$ large enough, $|\sigma_{\mathbf{A}}(x,\xi;k)| \geqslant
c\langle\xi\rangle^m$ in a neighborhood of $(x_0, \xi_0)$. The \emph{elliptic set} of an operator is denoted by $\Ell(\mathbf{A}) \subset \overline{T^*\Sigma}$ and the \emph{characteristic set} is the complement of the elliptic set. Operators can be inverted modulo $\mc{O}(k^{-\infty})$ smoothing operators on their elliptic set (parametrix construction).

%We now provide some examples:
%
%\begin{example} Let $X \in C^\infty(\Sigma,T\Sigma)$ be a vector field on $\Sigma$. For $\mathbf{k} := (k,\nabla) \in \mathbf{F}$, define
%\[
%\mathbf{A}_\mathbf{k} := h(\mathbf{k}) \nabla_X : C^\infty(L_0^k) \to C^\infty(\Sigma,L_0^k).
%\]
%Then its principal symbol is given by
%\[
%\sigma_{\mathbf{A}}(x,\xi;\mathbf{k}) = i (\xi,X(x)).
%\]
%\end{example}
%
%\begin{example} For $\mathbf{k} := (k,\nabla) \in \mathbf{F}$, define
%\[
%\mathbf{A}_\mathbf{k} := \tfrac{1}{2}h(\mathbf{k})^2\nabla^*\nabla : C^\infty(\Sigma,L_0^k) \to C^\infty(\Sigma,L_0^k).
%\]
%Then its principal symbol is given by
%\[
%\sigma_{\mathbf{A}}(x,\xi;\mathbf{k}) = \tfrac{1}{2}|\xi|^2_g.
%\]
%\end{example}
%
%\subsubsection{Semiclassical defect measures}

\subsection{Defect measures}

\label{ssection:defect-measures}

Quasimodes of energy $E \geq 0$ are sections $u_{k} \in C^\infty(\Sigma,L^k)$ satisfying
\begin{equation}
\label{equation:qm}
(k^{-2}\Delta_k-E) u_{k} = \mc{O}_{L^2}(\eps_{k}), \qquad \|u_k\|^2_{L^2}=1,
\end{equation}
where $\eps_{k} \to_{k \to \infty} 0$. Up to extraction along a subsequence $(k_n)_{n \geq 0}$, there exists a measure $\mu$ on $T^*\Sigma$ such that for all $a \in C^\infty_{{\comp}}(T^*\Sigma)$,
\begin{equation}
\label{equation:limit}
\lim_{n \to \infty} \langle\Op_{k_n}(a)u_{k_n},u_{k_n}\rangle_{L^2} = \int_{T^*\Sigma} a(x,\xi) d\mu(x,\xi).
\end{equation}
That $\Delta_k$ is elliptic on $\partial_\infty T^*\Sigma$ (boundary at
infinity of the radial compactification of $T^*\Sigma$), implies that $\mu$ is
a probability measure, see \cite[Appendix E, Exercise
23]{Dyatlov-Zworski-book}. In the following, recall that $p(x,\xi) = \tfrac{1}{2}|\xi|^2_g$.

\begin{proposition}
\label{proposition:scl-measure}
Let $\mu$ be a probability measure associated to the quasimodes \eqref{equation:qm}. Then the following holds:
\begin{enumerate}[label=\emph{(\roman*)}]
\item Assume that $\eps_k = o(1)$. Then
\[
\supp(\mu) \subset \mathrm{Char}(k^{-2}\Delta_k-E) = \{p=E\}.
\]
\item Assume that $\eps_k = o(1/k)$. Then $\mu$ is invariant by the magnetic Hamiltonian vector field $H_{p}^\Omega$, that is for all $a \in C^\infty_{{\comp}}(T^*\Sigma)$,
\[
\int_{T^*\Sigma} H_{p}^\Omega a(x,\xi) ~\dd\mu(x,\xi) = 0.
\]
\end{enumerate}
\end{proposition}

We refer to \cite[Appendix E.3]{Dyatlov-Zworski-book} for a proof.

\subsection{Technical estimates}

We will need some precise estimates on the remainder in the standard operations of semiclassical analysis (product, commutator, propagation, etc.). Recall that $(\Phi_t)_{t \in \R}$ is the magnetic flow on $T^*\Sigma$, that is the Hamiltonian flow generated by $H_p^\Omega$.

\begin{proposition} For any $E>0$, there exists a constant $C > 0$ such that for all $k \geq 0$, for all $a,b \in
C^\infty_{\comp}(T^*\Sigma)$ with $\mathrm{supp}(a), \mathrm{supp}(b) \subset
\{|\xi|_g \leq E\}$, the following holds: 
\begin{align}
\label{equation:calderon-fine}
\|\Op_k(a)\|_{L^2 \to L^2} & \leq \|a\|_{L^\infty(T^*\Sigma)} + k^{-1/2}\|a\|_{C^{14}(T^*\Sigma)} \\
\label{equation:product-fine}
\|\Op_k(a)\Op_k(b)-\Op_k(ab)\|_{L^2 \to L^2} &\leq Ck^{-1} \|a\|_{C^{17}(T^*\Sigma)}\|b\|_{C^{17}(T^*\Sigma)} \\
\label{equation:commutator-fine}
\|[\Op_k(a),\Op_k(b)]-\Op_k(-ik^{-1}\{a,b\})\|_{L^2 \to L^2} &\leq Ck^{-1}\|a\|_{C^{17}(T^*\Sigma)}\|b\|_{C^{17}(T^*\Sigma)} \\
\label{equation:egorov-fine}
\|e^{it k^{-1} \Delta} \Op_k(a) e^{-itk^{-1}\Delta}-\Op_k(a \circ \Phi_t)\|_{L^2 \to L^2} &\leq C k^{-1} \int_0^t \|a \circ \Phi_s\|_{C^{17}(T^*\Sigma)} \dd s.
\end{align}

%\begin{itemize}
%\item \textbf{\emph{Commutator.}}
%\begin{equation}
%
%\end{equation}
%
%\item \textbf{\emph{Egorov.}} For all $t \geq 0$:
%\begin{equation}
%
%\end{equation}
%
%
%\end{itemize}
\end{proposition}

The above estimates are very likely suboptimal by far; however, we did not try to optimize them.

\begin{proof}
Estimate \eqref{equation:calderon-fine} follows from \cite[Theorem 5.26]{Nonnenmacher-notes} applied in dimension $2$ and Lemma \cite[Lemma A.4]{Dyatlov-Jin-Nonnenmacher-22}. The estimates \eqref{equation:product-fine} and \eqref{equation:commutator-fine} can be found in \cite[Lemma A.6]{Dyatlov-Jin-Nonnenmacher-22} for the standard semiclassical quantization. It is an exercise to verify that the proofs go through for the twisted quantization. Finally, \eqref{equation:egorov-fine} follows from \eqref{equation:commutator-fine}. Indeed, since $e^{it k^{-1} \Delta}$ is unitary on $L^2(\Sigma, L^k)$, it suffices to estimate $\|B(t)\|$, where
\[
B(t) := \Op_k(a) - e^{-itk^{-1}\Delta}\Op_k(a\circ\Phi_t)e^{itk^{-1}\Delta}.
\]
Notice that $B(0)=0$ and using \eqref{equation:commutator-fine}:
\[
\begin{split}
B'(t) &= e^{-itk^{-1}\Delta}\left( k[ik^{-2}\Delta_k,\Op_k(a\circ\Phi_t)]-\Op_k(H_p a \circ\Phi_t) \right) e^{itk^{-1}\Delta} \\
& = \mc{O}_{L^2 \to L^2}(\|a\circ\Phi_t\|_{C^{17}(T^*\Sigma)} h),
\end{split}
\]
where $p$ is the principal symbol of $k^{-2}\Delta_k$. Writing
\[
B(t) = \int_0^t B'(s)\dd s,
\]
we find the claimed result.
% {\color{red}TL: Il y a une légère arnaque car on applique l'estimée \eqref{equation:commutator-fine} ne s'applique qu'avec des symboles à support dans $\{|\xi| \leq 10\}$ alors que $p$ n'est pas à support compact...}
\end{proof}

\section{Construction of eigenstates}

\label{section:constant}

%
%$ = -iB \vol$ and $B$ is constant. Given $k \in \Z_{\geq 0}$, there is an induced curvature $\nabla^{k}$ on $L^{k}$, whose curvature is $-ikB\vol$. We set $\Delta_k := \tfrac{1}{2}(\nabla^{k})^*\nabla^{k}$. The semiclassical limit $k \to \infty$ refers to understanding the behaviour of Laplace eigenfunctions of $k^{-2} \Delta_{k}$ on $C^\infty(\Sigma,L^{k})$.
%

We now assume that $(\Sigma,g)$ is a Riemannian surface of constant curvature $-1$, genus $\geq 2$, and $L \to \Sigma$ is a Hermitian line bundle equipped with a unitary connection $\nabla$ with curvature $F_\nabla = -iB\vol$ and $B > 0$ is constant. Without loss of generality, since $H^2(\Sigma,\Z) \simeq \Z$ and we will be
considering power $L^k \to \Sigma$, we can further assume that $\deg(L)=1$.
Note that by \eqref{equation:deg}, this forces $B = |\chi(\Sigma)|^{-1}$. Hence, in what follows, the magnetic intensity on $M$ \textbf{is fixed and equal to $B = |\chi(\Sigma)|^{-1}$} on $L \to \Sigma$. 
% Theorem \ref{theorem:main} is based on the construction of suitable Gaussian states. Most objects below depend on the semiclassical parameter $k$.

%{\color{red}TL: $h=\chi/k$ below?}
This section is organized as follows:
\begin{itemize}
\item In \S\ref{ssection:averaging}, we adapt Weinstein's averaging method \cite{Weinstein-77} to our context and introduce an operator $\mathbf{A}_k$ whose principal symbol generates a $2\pi$-periodic flow on $\{p < E_c\}$ which is a renormalization of the magnetic flow;
\item We then use it in \S\ref{ssection:quasimodes} to construct specific eigenstates concentrating in phase space on any closed orbit of the magnetic flow.
\end{itemize}

%(Equivalently, the first eigenvalues of $k^{-2}\Delta_{\mathbf{k}}$ are given by $\tfrac{1}{k \chi(\Sigma)}(m+\tfrac{1}{2})-\tfrac{m(m+1)}{2k^2}$).

\subsection{The averaging method} \label{ssection:averaging} By Proposition \ref{theo:Landau_Level}
applied to $L^k \to \Sigma$ with curvature $F_{\nabla^k}=-ikB
  \op{vol}$, the first eigenvalues of $\Delta_k = \tfrac{1}{2}\nabla^*\nabla : C^\infty(\Sigma,L^k) \to C^\infty(\Sigma,L^k)$ are given for $m < \lfloor kB\rfloor $ by
\begin{equation}
\label{equation:ev}
\lambda_{k,m} = kB (m+\tfrac{1}{2})-\tfrac{m(m+1)}{2}.
\end{equation}
Notice that $\lambda_{k, \lfloor kB \rfloor - 1 } \sim k^2 E_c$ as $k
\rightarrow \infty$, where $E_c
= \tfrac{1}{2} B^2$ is the critical energy. As a consequence, for any $\eps >0$, we have that 
 $$\op{Sp}
(k^{-2} \Delta_k ) \cap [0, E_c -\eps ] \subset  \{k^{-2}  \la_{k,m}  ~|~   m < \lfloor
kB\rfloor \}$$ when $k$ is sufficiently large.

\subsubsection{Weinstein's periodic operator} The following paragraph is
inspired by Weinstein's averaging method \cite{Weinstein-77}. Let $\Pi_{k,m}$ be the orthogonal projector of $C^\infty( M, L^k)$ onto
 $\ker(\Delta _k -\lambda_{k,m})$, and set:
\begin{equation}
\label{equation:a}
\mathbf{A}_k := k^{-1} \sum_{m =0 }^{ \lfloor kB \rfloor -1}  m \Pi_{k,m}.
\end{equation}
Observe that by construction $\op{Sp}(\mathbf{A}_k) \subset k^{-1} \Z_{\geq 0}$ and thus
\begin{equation}
\label{equation:cool}
e^{2i k \pi \mathbf{A}_k } = \mathbf{1}.
\end{equation}
Furthermore, by \eqref{equation:ev}, on the space
\[
\mc{I}_k := \bigoplus_{m=0}^{ \lfloor kB \rfloor -1}
\ker(\Delta_k-\lambda_{k,m}),
\]
the operator $\mathbf{A}_k$ satisfies the identity
\begin{equation}
\label{equation:relation}
k^{-2} \Delta_k = B(\mathbf{A}_k+ \tfrac{1}{2} k^{-1} )- \tfrac{1}{2} 
\mathbf{A}_k(\mathbf{A}_k+ k^{-1} ).
\end{equation}
 Equivalently, 
\begin{equation}
\label{equation:relation2}
\mathbf{A}_k= B-\tfrac{1}{2} k^{-1} - \sqrt{ B^2 - 2 k^{-2}
  \Delta_k + \tfrac{1}{4} k^{-2}},
\end{equation}
on $\mc{I}_k$. Finally, observe that
\begin{equation}
\label{eq:integrale_projecteur}
\Pi_{k,m} = \dfrac{1}{2\pi} \int_0^{2\pi} e^{-imt} e^{itk\mathbf{A}_k} \dd t.
\end{equation}
This turns out to be a Fourier Integral Operator in the semiclassical regime $k \to \infty$.

Let
\[
D^*\Sigma := \{(x,\xi) \in T^*\Sigma ~|~ |\xi| < B\}.
\]
Let $a$ be the function defined on $D^*\Sigma$ through the relation
\[
p(x,\xi) = \beta (a ( x, \xi)), \qquad \forall (x,\xi) \in D^*\Sigma,
\]
where $\beta : [0,B] \rightarrow [0, \frac{1}{2} B^2 ]$ is the function $ \beta (s)  :=  B s  -\frac{1}{2} s^2$. The following holds:

\begin{lemma}
Let $\chi \in C^\infty_{{\comp}}(T^*\Sigma)$ such that $\supp(\chi) \subset D^*\Sigma$ and $\varphi \in C^\infty(\R)$ such that $\supp(\varphi) \subset (-\infty,B)$. Then
\[
\Op_k(\chi) \mathbf{A}_k, \varphi(\mathbf{A}_k) \in \Psi^{\comp}_{\mathrm{tsc}}(\Sigma),
\]
and these operator have respective principal symbols $\chi a$ and $\varphi(a)$.
\end{lemma}

We emphasize that $\mathbf{A}_k$ is \emph{not} a pseudodifferential operator because of the spectral cutoff involved in its definition (the spectral projector onto $\mc{I}_k$). However, in what follows, we will always use $\Op_k(\chi) \mathbf{A}_k, \varphi(\mathbf{A}_k)$ which \emph{are} pseudodifferential operators. For simplicity, we also say that $a$ is the principal symbol of $\mathbf{A}_k$, although this only makes sense in $D^*\Sigma$.

\begin{proof}
We define $f_k : [0,\tfrac{1}{2}B^2] \to \R$ by:
\begin{equation}
\label{equation:fk}
f_k(s) = B-\tfrac{1}{2}k^{-1}-\sqrt{B^2-2s+\tfrac{1}{4}k^{-2}}.
\end{equation}
Notice that $\varphi \circ f_k$ is smooth with compact support in $[0,\tfrac{1}{2}B^2)$ by the assumption made on the support of $\varphi$; it can thus be extended to a smooth compactly supported function $c_k$ on $[0,\infty)$. In addition, using \eqref{equation:relation2}, the equality $\varphi(\mathbf{A}_k) =\varphi(f_k(k^{-2}\Delta_k)) = c_k(k^{-2}\Delta_k)$ holds. The function $c_k$ has compact support on $[0,\infty)$ and uniformly bounded derivatives as $k \to \infty$. Hence, by standard functional calculus (see \cite[Theorem 14.9]{Zworski-12} for instance), $c_k(k^{-2}\Delta_k) = \varphi(\mathbf{A}_k) \in \Psi^{\comp}_{\mathrm{tsc}}(\Sigma)$.

We now establish that $\Op_k(\chi) \mathbf{A}_k \in \Psi^{\comp}_{\mathrm{tsc}}(\Sigma)$. For that, we write
\[
\Op_k(\chi) \mathbf{A}_k =\Op_k(\chi) \psi(k^{-2}\Delta_k) \mathbf{A}_k + \Op_k(\chi) (\mathbbm{1}-\psi(k^{-2}\Delta_k)) \mathbf{A}_k,
\]
where $\psi $ is smooth with compact support in $(-\infty, E_c )$ and such that $\psi(p) \equiv 1$ on the support of $\chi$. Notice that $\Op_k(\chi) \psi(k^{-2}\Delta_k) \mathbf{A}_k = \Op_k(\chi) d_k(k^{-2}\Delta_k)$, where $d_k(x) = \psi(x) f_k(x)$ and $f_k$ was defined in \eqref{equation:fk}. The function $d_k$ extends to a smooth compactly supported function on $[0,\infty)$ and the same argument as in the previous paragraph using functional calculus shows that $\Op_k(\chi) \psi(k^{-2}\Delta_k) \mathbf{A}_k \in \Psi^{\comp}_{\mathrm{tsc}}(\Sigma)$. Let us now prove that $\Op_k(\chi) (\mathbbm{1}-\psi(k^{-2}\Delta_k)) \mathbf{A}_k$ is a residual operator, namely that it has smooth Schwartz kernel all of whose derivatives are $\mc{O}(k^{-\infty})$. By standard arguments, it suffices to prove that
\[
\Op_k(\chi) (\mathbbm{1}-\psi(k^{-2}\Delta_k)) \mathbf{A}_k : H^{-N}_k(\Sigma,L^k) \to H^{+N}_k(\Sigma,L^k)
\]
is bounded for all $N \geq 0$ with norm $\mc{O}(k^{-\infty})$, where the space $H^s(\Sigma,L^k)$ is defined as the completion of $C^\infty(\Sigma,L^k)$ with respect to the norm
\begin{equation}
\label{equation:sobolev-norm}
\|u\|^2_{H^s(\Sigma,L^k)} := \|(k^{-2}\Delta_k+\mathbbm{1})^{s/2} u\|^2_{L^2(\Sigma,L^k)}.
\end{equation}
By the support property of $\psi$, the operator $\Op_k(\chi) (\mathbbm{1}-\psi(k^{-2}\Delta_k))$ clearly satisfies this property (this is a composition of two operators with disjoint microsupport). Hence, it suffice to show that $\mathbf{A}_k : H^{-N}_k(\Sigma,L^k) \to H^{-N}_k(\Sigma,L^k)$ is bounded with norm $\mc{O}(1)$ for some integer $M \geq 0$. Using \eqref{equation:sobolev-norm}, this amounts to proving that
\[
(k^{-2}\Delta_k+\mathbbm{1})^{-N/2} \mathbf{A}_k (k^{-2}\Delta_k+\mathbbm{1})^{N/2} : L^2(\Sigma,L^k) \to L^2(\Sigma,L^k)
\]
has norm $\mc{O}(1)$. However, this is immediate using \eqref{equation:a} and
the fact that this operator acts diagonally on the Fourier decomposition into eigenmodes of $\Delta_k$.

Finally, to compute the principal symbol of these operators, merely observe
that by \eqref{equation:relation}, we have  $p = B a - a^2/2 = \beta(a)$.
\end{proof}
%
%Equivalently, it can be shown that for any $(x, \xi) \in T^* \Si$ with $| \xi | <B$, $a (x,\xi)$ is the integral of $\Om$ on a disk bounded by the (magnetic) orbit of
%$(x,\xi)$ (that is the action). Then we have
%\begin{gather} 
%  \begin{split} p( x,\xi) = \beta ( a ( x, \xi)) , \quad \forall | \xi | < B, \qquad \la_m = \beta ( m + \tfrac{1}{2} ) +
%       \tfrac{1}{8} , \quad \forall m = 0 , \ldots N-1,
%     \end{split}
%   \end{gather}

The function $a$ will play an important role in the sequel. One important property is the following:

\begin{lemma}
\label{lemma:2pi-periodic}
The Hamiltonian vector field $H_a^\Omega$ of the symbol $a$ generates a $2 \pi$-periodic flow on $D^*\Sigma$.
\end{lemma}

\begin{proof}This can be seen by writing $a = \al (p)$ with $\al(y) := B - \sqrt{ B^2 - 2 y }$, and noticing
that $\al' (E) = ( B^2 - 2 E)^{-\frac{1}{2}}$ is the period of the orbits of $H_p^\Omega$
with energy $p =E$ (see Proposition \ref{proposition:dynamic}, item (i)).
\end{proof}

Given a function $f \in C^\infty ( D^*\Sigma)$, introduce its average
$\langle f \rangle \in C^\infty (D^*\Sigma)$ by:
\begin{equation}
\label{equation:average}
\langle f \rangle (z) := \frac{1}{2\pi} \int_0 ^{2 \pi} f ( \Phi_t^a (z)) ~\dd t,
\end{equation}
where $(\Phi_t^a)_{t \in \R}$ is the $2\pi$-periodic Hamiltonian flow generated by $H_a^\Omega$ on $D^*\Sigma$:

\begin{lemma} \label{lem:commutateur_espace_propre}
Fix $\eps > 0$. For all $f \in C^\infty_{{\comp}} (D^*\Sigma)$, $k \geq 0$, for all $0 \leq m \leq (1-\eps)kB$, the following holds:
\[
\Pi_{k,m}\Op_{k}(f) \Pi_{k,m} = \Op_{k}(\langle f \rangle)
\Pi_{k,m} + \bigo_{\mc{L}(L^2)}( k^{-1}).
\] 
Here the $\bigo $ depends on $f$ but is uniform with respect to $m$.
\end{lemma}

The proof is a variation on Egorov's Theorem (see \cite[Theorem 15.2]{Zworski-12} for instance).

\begin{proof}
In the proof, the $\mc{O}$'s are measured in operator norm $L^2 \to L^2$. Fix
$f  \in C^\infty_{{\comp}}(D^*\Sigma)$ and choose $\varphi \in
C^\infty(-\infty , B )$
  such that $ \varphi (a) =1$ on a neighbourhood of the support of $f$. Then
\[
\Op_k( f) \varphi (\mathbf{A}_k ) \equiv \Op_k(f) \equiv \varphi (\mathbf{A}_k ) \Op_k
  (f),
  \]
 where $\equiv$ stands for equality modulo $\bigo ( k^{-\infty})$. This yields
\begin{xalignat*}{2}
 [ \Op_k(f) , \mathbf{A}_k ] & \equiv [ \Op_k(f) , \varphi ( \mathbf{A}_k ) \mathbf{A}_k ]
  = \tfrac{1}{ik} \Op_k ( \{ f, \varphi (a) a \} ) + \bigo (k^{-2}) \\ & =
  \tfrac{1}{ik} \Op_k ( \{ f, a \} ) + \bigo (k^{-2}). 
\end{xalignat*}
Write $U_t := e^{itk\mathbf{A}_k}$. Integrating the previous relation, we obtain the weak form of Egorov's theorem for $\mathbf{A}_k$:
\begin{gather} \label{eq:egorov}
  U_t \Op_k(f) U_{-t} = \Op_k( f \circ \Phi_t^a ) + \bigo (k^{-1}).
\end{gather}
Since $U_t \Pi_{k,m} = \Pi_{k,m} U_t = e^{itm} \Pi_{k,m}$, we obtain that
\[
\Pi_{k,m} \Op_k(f) \Pi_{k,m} = \Pi_{k,m} U_t \Op_k(f) U_{-t} \Pi_{k,m} = \Pi_{k,m} \Op_k(f \circ
\Phi_t^a) \Pi_{k,m} + \bigo ( k^{-1}).
\]
Averaging over time, we find that
\begin{gather}
\label{eq:moyenne}
  \Pi_{k,m} \Op_k(f) \Pi_{k,m} = \Pi_{k,m} \Op_k(\langle f \rangle ) \Pi_{k,m} + \bigo (k^{-1}).
\end{gather}
Moreover, since $\langle f \rangle \circ \Phi_t^a = \langle f \rangle $, we
deduce from \eqref{eq:egorov} that $ [ U_t, \Op_k( \langle f \rangle ) ] =
\bigo ( k^{-1})$ and from \eqref{eq:integrale_projecteur} that $[\Pi_{k,m} , \Op_k(
\langle f \rangle)] = \bigo ( k^{-1})$. The claim then follows from this last
relation and \eqref{eq:moyenne}.
\end{proof}

\subsubsection{Normal form for a perturbed operator} We conclude by establishing a normal form for operators $\varphi(\mathbf{A}_k) + k^{-1}R$ which will be used in the proof of Theorem \ref{theorem:constant2}.

\begin{lemma}
\label{lemma:normal-form}
Let $R \in \Psi^{{\comp}}_{\mathrm{tsc}}(\Sigma)$ be formally
selfadjoint with microsupport in $D^*\Sigma$ and $\varphi \in
C^\infty_{{\comp}}(0,\infty)$ such that $\varphi'(a) \neq 0$ on the microsupport of $R$.
Then there exists a sequence $(S_j)_{j \geq 0}$ with $S_j \in
\Psi^{{\comp}}_{\mathrm{tsc}}(\Sigma)$ and $S_j$ selfadjoint, such that
for all $\ell \geq 2$, there exists $R'_\ell \in
\Psi^{{\comp}}_{\mathrm{tsc}}(\Sigma)$ selfadjoint such that 
\begin{equation}
\label{equation:prod}
U_{\ell}^* (\varphi(\mathbf{A}_k) + k^{-1}R)  U_{\ell}
 = \varphi(\mathbf{A}_k) + k^{-1}R'_\ell + \mc{O}_{\Psi_{\mathrm{tsc}}^{{\comp}}(\Sigma)}(k^{-\ell}),
\end{equation}
with $[\varphi(\mathbf{A}_k),R'_\ell]=0$ and $U_{\ell} = \prod_{j=0}^{\ell-2}
e^{-ik^{-j}S_j}$.
\end{lemma}

Given an operator $R$ with microsupport in $D^*\Sigma$, set
\[
\langle R\rangle := \dfrac{1}{2\pi} \int_0^{2\pi} e^{itk\mathbf{A}_k} R e^{-itk\mathbf{A}_k} ~ \dd t.
\]
Notice that $\langle R\rangle \in \Psi^{\comp}_{\mathrm{tsc}}(\Sigma)$ and $\langle R \rangle$ commutes with $\mathbf{A}_k$ since
\[
\begin{split}
[\mathbf{A}_k,  \langle R \rangle ] &= \dfrac{1}{2\pi} \int_0^{2\pi} e^{itk\mathbf{A}_k} [\mathbf{A}_k,R]e^{-itk\mathbf{A}_k} ~ \dd t \\
& =  \dfrac{1}{2\pi} \int_0^{2\pi} \dfrac{d}{dt} (e^{itk\mathbf{A}_k} R e^{-itk\mathbf{A}_k}) ~ \dd t = \dfrac{1}{2\pi}(e^{2i\pi k\mathbf{A}_k} Re^{-2i \pi k\mathbf{A}_k}-R) = 0,
\end{split}
\]
using \eqref{equation:cool}. Furthermore, if $R \in
\Psi_{\mathrm{tsc}}^{\comp}(\Sigma)$ has microsupport in $D^*\Sigma$, then $\langle R\rangle \in
\Psi_{\mathrm{tsc}}^{\comp}(\Sigma)$ has microsupport in $D^*\Sigma$ too, and has principal symbol given by $\langle
r\rangle$, where $r$ is the principal symbol of $R$. In addition $\langle
R\rangle$ is selfadjoint if $R$ is. Finally, let us observe that for any function $c \in C^\infty_{{\comp}}(D^*\Sigma)$, there exists a function $s \in C^\infty_{{\comp}}(D^*\Sigma)$ such that $\{a,s\}^\Omega = c-\langle c\rangle$. For that, it suffices to set
\begin{equation}
\label{equation:s}
s := \dfrac{1}{2\pi} \int_0^{2\pi}\left(\int_0^t (c-\langle c\rangle) \circ \Phi_{t'}^a~ \dd t' \right)\dd t.
\end{equation}
This is well-defined by $2\pi$-periodicity of the Hamiltonian flow of $a$.
More generally, if $\varphi \in C^\infty(\R)$ satisfies $\varphi'(a) \neq 0$ on $D^*\Sigma$ then one can find a solution $s \in C^\infty_{\comp}(D^*\Sigma)$ to
\begin{equation}
\label{equation:resoluble}
\{\varphi(a),s\}^\Omega = c-\langle c\rangle,
\end{equation}
using $\{\varphi(a),s\}^\Omega = \varphi'(a)\{a,s\}^\Omega$ and \eqref{equation:s}.

\begin{proof}
We use the notation $h:=1/k$. We argue by induction on $\ell \geq 2$ and start with $\ell=2$. Notice that 
\[
\begin{split}
e^{iS_0}(\varphi(\mathbf{A}_k) + hR)e^{-iS_0} & = \varphi(\mathbf{A}_k) + hR + i[S_0,\varphi(\mathbf{A}_k)] + \mc{O}(h^{2}) \\
& = \varphi(\mathbf{A}_k) + hR - h\Op_k(\{\varphi(a),s_0\}^\Omega) + \mc{O}(h^{2})
\end{split}
\]
By \eqref{equation:resoluble}, we can find a symbol $s_0 \in C^\infty_{{\comp}}(D^*\Sigma)$ such that $\{\varphi(a),s_0\}^\Omega = r-\langle r\rangle$, where $r \in C^\infty_{{\comp}}(D^*\Sigma)$. Since $s_0$ is real-valued, we can find $S_0$ selfadjoint with principal symbol $s_0$. Going back to the previous equation, we then find
\[
e^{iS_0}(\varphi(\mathbf{A}_k) + hR)e^{-iS_0}  =  \varphi(\mathbf{A}_k) + h\langle R\rangle + \mc{O}(h^{2}).
\]
Notice that $\langle R \rangle$ is selfadjoint. This proves the claim for $\ell=2$ using $[\varphi(\mathbf{A}_k),\langle R\rangle]=0$ as $[\mathbf{A}_k,\langle R\rangle]=0$.

Assume now that \eqref{equation:prod} holds for $\ell \geq 2$, that is 
\[
 U_{\ell}^* (\varphi(\mathbf{A}_k) + hR) U_{\ell}  = \varphi(\mathbf{A}_k) + hR'_\ell + h^{\ell} C,
\]
where $C \in \Psi_{\mathrm{tsc}}(\Sigma)$ is selfadjoint. Multiplying by $e^{i h^{\ell-1}S_{\ell-1}}$ on the left and $e^{-i h^{\ell-1}S_{\ell-1}}$ on the right, we find
\[
\begin{split}
U_{\ell +1}^* (\varphi(\mathbf{A}_k) + hR)  U_{\ell+1}  & = \varphi(\mathbf{A}_k) + hR'_\ell + h^{\ell} C  + i h^{\ell-1}[S_{\ell-1}, \varphi(\mathbf{A}_k)] + \mc{O}(h^{\ell+1}).
\end{split}
\]
By \eqref{equation:resoluble} again, we can find $s_{\ell-1} \in C^\infty_{{\comp}}(D^*\Sigma)$ such that $\{\varphi(a),s_{\ell-1}\}^\Omega = c-\langle c\rangle$. This yields:
\[ U_{\ell +1}^* (\varphi(\mathbf{A}_k) + hR)  U_{\ell+1}   = \varphi(\mathbf{A}_k) + h(\underbrace{R'_\ell + h^{\ell-1} \langle C\rangle}_{:=R'_{\ell+1}}) + \mc{O}(h^{\ell+1}),
\]
where $S_{\ell-1}$ is chosen selfadjoint with principal symbol $s_{\ell-1}$. Notice that $R'_{\ell +1}$ is selfadjoint. This proves the claim.
\end{proof}

\subsection{Localized eigenstates} \label{ssection:quasimodes} We now construct specific eigenstates for the operator $k^{-2}\Delta_k$ with adapted localization properties.

\subsubsection{Energy $E=0$} We first investigate the low energy eigenstates in the semiclassical regime
\[
k^{-2}\Delta_k u_k = E_k u_k, \qquad u_k \in C^\infty(\Sigma,L^k),\; \|u_k\|_{L^2}=1,\; E_k = \mc{O}(1/k).
\]
This corresponds the regime of eigenvalues $\lambda_{k,m}$ with $k \to \infty$ and $m \leqslant C$, where $C > 0$ is a fixed constant. From now on, we use the notation $h = 1/k$. Let $\underline{x}  \in \Si$, $u \in L_{\underline{x}}$ with $|u|=1$ and define the Dirac section of
$L^k$ by $(\delta_u^k, s) = s(\underline{x} )/u^k$ for any $k \geq 0$, $s \in C^\infty (\Sigma,L^k)$. We introduce the section
\begin{gather} \label{eq:def_etat_coherent}
\mathbf{e}_{k,m,\underline{x}} := \Pi_{k,m} \delta_u^k.
\end{gather}

The following holds:
%{\color{red}TL: $h=1/k$ ou $h=\chi/k$?}

\begin{proposition}
\label{proposition:e0}
Let $C > 0$ and $f \in C^\infty_{{\comp}}(T^* \Sigma)$. Then for all $k \geq 0$, $0 \leq m \leq C$, $\mathbf{e}_{k,m,\underline{x}} \in C^\infty(\Sigma,L^k)$ satisfies
  \begin{xalignat}{2}
\begin{split} \label{eq:estim1}
& \| \mathbf{e}_{k,m,\underline{x}}  \|^2 = \tfrac{k}{2\pi}  ( 1 + \bigo ( k^{-1/2})), \\
&  \langle \Op_k(f)  \mathbf{e}_{k,m,\underline{x}} , \mathbf{e}_{k,m,\underline{x}} \rangle = \| \mathbf{e}_{k,m,\underline{x}}  \|^2 \bigl( f (\underline{x},0) + \bigo
(k^{-1/2}) \bigr).
\end{split}
\end{xalignat}
%Moreover, the $\bigo$'s are uniform with respect to $(k,
%\nabla) \in \mathbf{F}$. 
\end{proposition}

As we will see, $\mathbf{e}_{k,m,\underline{x}}$ is a Gaussian state centered at
$\underline{x}$. We provide a proof building on \cite{Landau1,Landau2}.
%The interested reader can also consult the second article \cite{Charles-Lefeuvre-25-2}, where a self-contained proof of this fact will be %provided.

\begin{proof}
The section $ \mathbf{e}_{k,m,\underline{x}}$ is given in terms of the Schwartz kernel of
  the projector $\Pi_{k,m}$ by
\begin{gather} \label{eq:noyau_coherent}
  \mathbf{e}_{k,m,\underline{x}} (x) \otimes \con{u}^k = \Pi_{m,k} (x, \underline{x})
  \in L_x^k \otimes \con{L}_{\underline{x}}^k.
\end{gather} 
% So we have
% \begin{gather} 
% \langle \Op_k(f)  \mathbf{e}_{k,m,\underline{x}} ,
% \mathbf{e}_{k,m,\underline{x}} \rangle = (\Pi_{m,k} \op{Op}_k ( f) \Pi_{m,k} )(
% \underline{x}, \underline{x} )
% \end{gather}
It was proved in \cite{Landau1,Landau2} that $\Pi_{k,m} (x,\underline{x})$ has an
asymptotic expansion, which leads to the following expansion for $ \mathbf{e}_{k,m,\underline{x}}$. Introduce coordinates $(x_1, x_2)$ centered at $\underline{x}$ and such that the metric $g$
is $(1 + \bigo ( |x|) ) (dx_1^2 + dx_2^2)$ with $|x|^2 =  x_1^2 + x_2^2$. Choose a primitive $\al = \al_1
dx_1 + \al_2 dx_2 $ of the Riemannian volume $\vol$ on a neighborhood of $y$, such that $ \al_1 (\underline{x})
= \al_2 (\underline{x}) =0$ and $\partial_{x_i} \al_j (\underline{x})$ is antisymmetric. Let $t$ be a
frame of $L$ such that  $t (\underline{x}) = u$ and $\nabla t = \tfrac{1}{i}
\al \otimes t$. Then one has on a neighborhood of $\underline{x}$:
\begin{gather} \label{eq:theta}
 \mathbf{e}_{k,m,\underline{x}} (x) =  \frac{k}{2 \pi} e^{-\frac{1}{4} k  |x|^2 } \sum_{\ell
   =0}^N k^{ -\frac{\ell}{2} } a_\ell ( k^{\frac{1}{2} } x_1, k^{\frac{1}{2}}
     x_2)  t^k(x)  + \bigo ( k^{- \frac{N+1}{2}})  
\end{gather}
where the $\bigo $ is in the $\mathcal{C}^0$-norm and the $a_{\ell}$'s are polynomials depending on $m$ and smoothly on
$\underline{x}$.  The leading coefficient $a_0$ is given in terms of the $m$-th
Laguerre polynomial $Q_m$ by
\begin{gather} \label{eq:leading}
a_0 (x_1, x_2)  = Q_m ( |x|^2), \qquad Q_m (t) = \tfrac{1}{m!} ( \tfrac{d}{dt}
-1 ) ^m t^m
\end{gather}
Moreover, $\mathbf{e}_{k,m,\underline{x}}$ is $\bigo_{C^\infty} ( k^{-\infty})$ on any compact set not
containing $\underline{x}$.
Formulas \eqref{eq:theta}, \eqref{eq:leading} correspond to \cite[Equations
(26) and (40)]{Landau1}.

Since $\Pi_{k,m}$
is a projector, $\| \mathbf{e}_{k,m,\underline{x}} \|^2 =  \Pi_{k,m}
(\underline{x},\underline{x})$, so by \eqref{eq:theta}, 
$$\| \mathbf{e}_{k,m,\underline{x}} \|
^2 =  \frac{k}{2\pi} + \bigo ( k^{1/2}) $$
  because $Q_m  (0) =1$. 
  We then compute the asymptotic expansion of the scalar product
  $$\langle \Op_k(f) \mathbf{e}_{k,m,\underline{x}},
  \mathbf{e}_{k,m,\underline{x}} \rangle = \sum_{\ell, \ell' \geqslant 0 }
k^{-\frac{1}{2} ( \ell + \ell')} I_{\ell, \ell'}(f,k) $$ with $I_{\ell, \ell'} (f,k)$ given by the
integral 
$$  \Bigl( \frac{k}{2 \pi} \Bigr) ^4   \int  e^{ik ( x- y ) \xi - \frac{1}{4}k ( |x|^2 + |y^2| ) }
f ( \tfrac{1}{2} (x+y), \xi ) a_{\ell} ( k^{\frac{1}{2}} x) \overline{
  a_{\ell'} ( k^{\frac{1}{2} } y )}  \; \dd\xi_1 \dd\xi_2 \dd x_1 \dd x_2 \dd y_1 \dd y_2 $$
Notice that the phase is quadratic, non degenerate. It follows from a stationary phase computation that $I_{\ell, \ell'} ( f,k) = k C_{\ell, \ell'} f(0, 0 ) + \bigo
( k^{\frac{1}{2}} )$ for some complex coefficient $C_{\ell, \ell'}$.  This yields:
$$\langle \Op_k(f) \mathbf{e}_{k,m,\underline{x}}, \mathbf{e}_{k,m,\underline{x}} \rangle
= k f(0,0) C_{0,0} + \bigo ( k^{1/2})$$
and $C_{0,0} = (2 \pi)^{-1} $ because of the previous estimate of $\| \mathbf{e}_{k,m,\underline{x}}
\|^2$. This concludes the proof. We could actually slightly improve the result
by replacing the $\bigo (k^{-1/2})$'s in \eqref{eq:estim1} by $\bigo ( k^{-1})$'s, which follows
from the fact that 
each $a_\ell$ has the same parity as $\ell$.  
\end{proof}

\subsubsection{Energies $0 < E < E_c$}  We now investigate the eigenstates such that
\[
k^{-2}\Delta_{k} u_k = (E+o(1))u_k,
\]
with $0 < E < E_c$. Equivalently, this corresponds to the regime of
eigenvalues $\lambda_{k,m}$ with $k \to \infty$ and  $\eps k B \leqslant m \leqslant (1- \eps)kB$, where $\eps> 0$ is fixed.

\begin{proposition} \label{lem:gaussian-beam}
Let $z_0 \in D^*_0\Sigma := D^*\Sigma\setminus \{ \xi = 0 \}$. Then, there exists a neighborhood $U \subset D^*_0\Sigma$ of $z_0$ such that for all $f \in C^\infty_{{\comp}}(T^*\Sigma)$, for all $k \geq 0$, $\eps kB \leq m \leq (1-\eps)kB$, and $z \in U$ such that
\[
m=ka(z),
\]
there exists $\mathbf{e}_{k,m,z} \in C^\infty(\Sigma,L^{k})$ such that $\| \mathbf{e}_{k,m,z} \|^2 = 1$ and the
projection
\[
\mathbf{f}_{k,m,z} := \Pi_{k,m} \mathbf{e}_{k,m,z}
\]
satisfies
\begin{xalignat}{2}
\begin{split} 
& \| \mathbf{f}_{k,m,z} \|^2 = \tfrac{1}{ 2 \sqrt \pi} k^{-\frac{1}{2}}(1+o_{f,U}(1)), \\
& \langle \Op_k(f)  \mathbf{f}_{k,m,z}, \mathbf{f}_{k,m,z} \rangle = \|\mathbf{f}_{k,m,z}\|^2 \bigl( \langle f \rangle (z) + \bigo_{f,U}
(k^{-\frac{1}{4}}) \bigr).
\end{split}
\end{xalignat}
The remainder terms $\bigo$'s are uniform with respect to $k \geq 0$.
\end{proposition}

The state $\mathbf{e}_{k,m,z}$ is a coherent state centred at $z$ and $\mathbf{f}_{k,m,z}$
is a Gaussian beam supported on the magnetic geodesic of $z$ (i.e. the flowline of $H_a^\Omega$ passing through $z$).

%For the proofs, we can restrict everything to $\mathcal{H} = \op{Im} 1_{ ]-\infty,
%  \frac{1}{2} [} ( h^2 \Delta)$ which corresponds semi-classically to the
%  subset $D$. Introduce the bounded operators of $\mathcal{H}$
%\begin{gather} 
%  \hat{A} = \sum_m hm \Pi_m,  \\ U_t = \exp \bigl(
%  it h^{-1} \hat{A} \bigr) = \sum_m
%  e^{itm} \Pi_m 
%\end{gather}
%and notice the formula
%\begin{gather} 
%  \Pi_m = \frac{1}{2\pi} \int_0 ^{2 \pi} e^{-itm} U_t \; dt
%\end{gather}
%By Theorem \ref{theo:Landau_Level}, we have that $ h^2 \la_m = \beta ( h ( m
%+\frac{1}{2} )) + \frac{1}{8} h^2$ where $\beta (x) = x - \frac{1}{2} x^2$ is
%the inverse of $\al$. Thus $hm = \al ( h^{2}  \la_m - \frac{1}{8} h^2 ) -
%\frac{1}{2} h $, so
%$$ \hat{A} = \al ( h^2 \Delta - \tfrac{1}{8} h^2) - \tfrac{1}{2} h $$
%Using that $ A = \al \circ H$ and $H$ is the symbol of $h^2 \Delta$, we deduce
%from the functional calculus of pseudodifferential operators that for any
%$\varphi \in C^\infty_0 ( ]-\infty, 1 [)$, $\varphi ( \hat{A} ) $ is a
%pseudodifferential operator of $\Psi_{\op{tsc}} ^ {-\infty} ( \Si) $ with symbol
%$\varphi (A)$. In this sense, $\hat{A}$ may be considered as a twisted pseudodifferential
%operator with symbol $A$.

\begin{proof}
We use $h=1/k$ in the proof once again. Let $\mathbf{e}_{k,m,z} \in C^\infty(\Sigma,L^k)$ be a Gaussian state centred at $z$ and such that $\|\mathbf{e}_{k,m,z}\|_{L^2}=1$ (see \cite[Example 1, Chapter 5]{Zworski-12} for instance). A quick computation using the stationary phase lemma reveals
  \begin{gather} \label{eq:etat_coherent}
    \Op_k(g) \mathbf{e}_{k,m,z}  = g (z) \mathbf{e}_{k,m,z} + \bigo_{L^2} (
  h^{\frac{1}{2} } ) 
\end{gather}
for any $g \in C^\infty_{{\comp}}(T^*\Sigma)$. Then $\mathbf{f}_{k,m,z} := \Pi_{k,m} \mathbf{e}_{k,m,z}$ (with $m=ka(z)$) satisfies by Lemma
\ref{lem:commutateur_espace_propre} that
\begin{xalignat}{2} \notag
  \langle \Op_k( f ) \mathbf{f}_{k,m,z}, \mathbf{f}_{k,m,z} \rangle & = \langle \Op_k( \langle f \rangle )
  \mathbf{e}_{k,m,z}, \Pi_{k,m} \mathbf{e}_{k,m,z} \rangle + \bigo ( h \|\mathbf{f}_{k,m,z}\| ) \\
  & = \langle f \rangle (z) \|\mathbf{f}_{k,m,z}\|^2 + \bigo ( h^{\frac{1}{2}} \|\mathbf{f}_{k,m,z}\|    ).
  \end{xalignat}
  To estimate the norm of $\mathbf{f}_{k,m,z}$, we use
  \eqref{eq:integrale_projecteur}, namely (recall that $U_t = e^{itk\mathbf{A}_k}$):
\begin{gather} \label{eq:int_norme}
\begin{split}
  \|\mathbf{f}_{k,m,z}\|^2 = \langle \Pi_{k,m} \mathbf{e}_{k,m,z}, \mathbf{e}_{k,m,z} \rangle & = \frac{1}{2 \pi }
  \int_0^{2 \pi } e^{-itm } \langle U_t \mathbf{e}_{k,m,z} ,\mathbf{e}_{k,m,z}
  \rangle \; \dd t \\
  & =  \frac{1}{2 \pi } \int_0^{2 \pi }  \langle  e^{itk(\mathbf{A}_k-a(z))}
  \mathbf{e}_{k,m,z} ,\mathbf{e}_{k,m,z} \rangle \; \dd t,
 \end{split}
\end{gather}
where we used that $m = k a(z)$. 
By assumption, the microsupport of $\mathbf{e}_{k,m,z}$ is $\{z\}$ (that is for
  any $\varphi \in C^\infty _{{\comp}} ( T^* \Si )$ identically equal to 1 on a
  neighborhood of $z$, $\Op_k( \varphi ) \mathbf{e}_{k,m,z} = \mathbf{e}_{k,m,z} + \bigo_{C^\infty} (
  h^{\infty})$). Then by Egorov's theorem, the microsupport of
  $U_t \mathbf{e}_{k,m,z}$ is $ \{ \Phi_t^a(z) \}$. Since $\Phi_t^a (z) = z $ if and only if $t$
  is a multiple of $2 \pi$, the integral
  \eqref{eq:int_norme} is modified by a $\bigo ( h^{\infty})$ if we
  restrict the integral to a neighborhood of $t=0$. 
  
  Microlocally, $\mathbf{A}_k-a(z)$ can be conjugated by a
  Fourier Integral Operator to $-ih\partial_{x_1}$ acting on
  $\R^2_{x_1, x_2}$, see \cite[Theorem 12.3]{Zworski-12}. Since we can restrict the integral \eqref{eq:int_norme} to a
  neighborhood of the origin, it is sufficient to do the computation in this
  model. The propagator is then given by $e^{itk(\mathbf{A}_k-a(z))} \Psi ( x_1, x_2) = \Psi ( x_1 +t, x_2)$ for $\Psi \in C^\infty(\R^2)$. The point $z$ is in the characteristic of $\mathbf{A}_k-a(z)$; it is therefore mapped to $(\underline{x}, \underline{\xi}_1=0,\underline{\xi}_2) \in T^*\R^2$. Up to conjugating again by a translation, we can further assume that $\underline{x}=0$. Let
  $\Psi_{\underline{x}, \underline{\xi}} \in C^\infty ( \R^2)$ be 
  the coherent state
  $$ \Psi_{\underline{x}, \underline{\xi}} ( x ) = ( \pi h ) ^{-1/2} e^{ -
    \frac{1}{2h} | x - \underline{x}|^2 + \frac{i}{h} \underline{\xi} ( x - \underline{x}) } = (\pi h)^{1/2}e^{ -
    \frac{1}{2h} | x|^2 + \frac{i}{h} \underline{\xi}_2 x_2 }.$$
By an exact computation, we find that for $\delta > 0$, $\chi \in C^\infty_{{\comp}}(-\delta,\delta)$ such that $\chi \equiv 1$ in a neighborhood of $0$,
\[
\begin{split} \frac{1}{2 \pi } \int_0^{2 \pi }  & \langle  e^{itk(\mathbf{A}_k-a(z))} \mathbf{e}_{k,m,z} ,\mathbf{e}_{k,m,z} \rangle_{L^2} \chi(t)  \dd t \\
  & = \frac{1}{2\pi} \int_{-\infty}^{+\infty} \int_{\R^2} \Psi _{ \underline{x} ,
  \underline{\xi}  } ( x_1 + t, x_2 ) \con { \Psi_{  \underline{x}, \underline{\xi} } (x_1,
  x_2)} \chi(t) \; \dd x_1 \dd x_2 \dd t \\
  & = \frac{ h^{1/2}}{2 \sqrt \pi }(1+ \mc{O}(h^\infty)).
  \end{split}
  \]
Going back to \eqref{eq:int_norme}, we find that $ \|\mathbf{f}_{k,m,z}\|^2_{L^2} =
\tfrac{1}{ 2 \sqrt \pi} h^{1/2} + \bigo ( h^{\infty})$, which
concludes the proof.
\end{proof}

\section{Proof of the main results}

\label{section:proofs}

In this final section, we prove Theorems \ref{theorem:main}, \ref{theorem:constant2} and conclude with the proof of Theorem \ref{theorem:horocyclic}. Up to a preliminary rescaling, we can always assume that $B=1/|\chi(\Sigma)|$.

\subsection{Proof of Theorem \ref{theorem:main}} We can now characterize the semiclassical defect measures in the three regimes.

\begin{proof}[Proof of Theorem \ref{theorem:main}] (i) Suppose $E=0$. We claim that any probability measure $\mu$ on $\Sigma$ is a semiclassical limit of eigenstates \eqref{equation:eigenstates}. First, given $x \in \Sigma$,
\[
\mathbf{e}'_{k,m,x} := \mathbf{e}_{k,m,x}/\|\mathbf{e}_{k,m,x}\|_{L^2}
\]
converges semiclassically to $\delta_x$ by Proposition \ref{proposition:e0}. Given $x_1, ..., x_N \in \Sigma$, $\mathbf{E}_{k,m,N} := N^{-1/2}(\mathbf{e}'_{k,m,x_1} + ... + \mathbf{e}'_{k,m,x_N})$ converges semiclassically to $N^{-1}(\delta_{x_1}+...+\delta_{x_N})$ and $\|\mathbf{E}_{k,m,N}\|^2_{L^2} = 1 + \mc{O}(k^{-\infty})$. Indeed, to compute the norm, let $\varphi_{i} \in C^\infty(\Sigma)$ be a small bump function centred at $x_i$, equal to $1$ on a small neighborhood of $x_i$ and such that $x_j \notin \supp(\varphi_i)$ for $j \neq i$. It is then immediate to verify that $\varphi_i \mathbf{e}'_{k,m,x_i} = \mathbf{e}'_{k,m,x_i} + \mc{O}_{C^\infty}(h^\infty)$ and $\varphi_j \mathbf{e}'_{k,m,x_i} = \mc{O}_{C^\infty}(h^\infty)$. This leads to
\[
\begin{split}
\|\mathbf{E}_{k,m,N}\|^2_{L^2} & = N^{-1}\langle \mathbf{e}'_{k,m,x_1} + ... + \mathbf{e}'_{k,m,x_N}, \mathbf{e}'_{k,m,x_1} + ... + \mathbf{e}'_{k,m,x_N}\rangle_{L^2} \\
& = 1 + N^{-1} \sum_{i \neq j} \langle \mathbf{e}'_{k,m,x_i},\mathbf{e}'_{k,m,x_j}\rangle_{L^2} \\
& =  1 + N^{-1} \sum_{i \neq j} \langle \varphi_i \mathbf{e}'_{k,m,x_i},\mathbf{e}'_{k,m,x_j}\rangle_{L^2} + \mc{O}(h^\infty) \\
& = 1 + N^{-1} \sum_{i \neq j} \langle  \mathbf{e}'_{k,m,x_i},\varphi_i \mathbf{e}'_{k,m,x_j}\rangle_{L^2} + \mc{O}(h^\infty)  = 1 + \mc{O}(h^\infty).
\end{split}
\]
The same computation shows that the associated defect measure is equal to $N^{-1}(\delta_{x_1}+...+\delta_{x_N})$. Finally, since Dirac masses are dense in probability measures, a diagonal extraction allows to construct a family $(\mathbf{E}_{k_n,m_n,N_n})_{n \geq 0}$ with microlocal limit $\mu$, and that for any probability measure $\mu$ on $\Sigma$.

We now assume that $0 < E < E_c$. Fix $z \in \{p=E\}$ and let $(z_k)_{k \geq 0}$ be a sequence converging to $z$ such that $a(z_k) = k^{-1}m_k$ with $m_k \in \Z_{\geq 0}$. By Proposition \ref{lem:gaussian-beam}, $\mathbf{f}_{k,m_k,z_k}/\|\mathbf{f}_{k,m_k,z_k}\|_{L^2}$ converges microlocally to $\delta_z$, the Dirac masses carried by the periodic orbit of $z$. As above, one can also realize any convex linear combination of such Dirac masses. As these are dense in flow-invariant probability measures on $\{p=E\}$, this proves the claim. \\

(ii) This follows immediately from Proposition \ref{proposition:scl-measure}, item (ii), and the unique ergodicity of the horocyclic flow (see Lemma \ref{lemma:horocycle}, item (ii)). \\

(iii) The proof follows \emph{verbatim} the standard proof of Quantum Ergodicity. In the framework of twisted quantization, this question was recently adressed in \cite[Section 5.3]{Cekic-Lefeuvre-24} (in a more general framework).
\end{proof}

\begin{remark} In the first regime, our proof shows a slightly
stronger result: for any invariant probability measure $\mu$, for any sequence
$(E_k)_{k \geq 0 }$ converging to $E \in [0, E_c]$ such that $E_k$ is in the spectrum of
$k^{-2} \Delta_k$ (and $E_k= \bigo ( k^{-1})$ when $E=0$), there exists a
sequence of eigenstates $(u_k)_{k \geq 0}$ associated to the eigenvalues $E_k$ such that $u_{k}
\rightharpoonup_{k \to \infty} \mu$.
\qed \end{remark}

\subsection{Proof of Theorem \ref{theorem:constant2}} \label{ssection:proof2}The proof of Theorem \ref{theorem:constant2} relies on the normal form established in Lemma \ref{lemma:normal-form}.

\begin{proof}[Proof of Theorem \ref{theorem:constant2}]
Let $(u_k)_{k \geq 0}$ be a sequence of eigenfunctions satisfying (recall $E_c=\tfrac{1}{2}B^2$):
\begin{equation}
\label{equation:pouet}
(k^{-2}\Delta_k + k^{-2}V) u_k = E_k u_k, \qquad 0 \leq E_k \leq E_c -\eps,
\; \|u_k\|_{L^2}=1,
\end{equation}
for some $\eps > 0$.

The multiplication operator $V \in \mc{L}(L^2)$ has norm $\|V\|_{\mc{L}(L^2) } \leq \|V\|_{L^\infty}$. As a consequence, 
\begin{equation}
\label{equation:confinement}
\mathrm{Sp}(\Delta_k + V) \subset \mathrm{Sp}(\Delta_k) + [-\|V\|_{L^\infty},\|V\|_{L^\infty}].
\end{equation}
By \eqref{equation:ev}, the gap between two consecutive eigenvalues of $\Delta_k$ (below the critical energy) is
\[
\lambda_{k,m}-\lambda_{k,m-1} = kB-m, \qquad 0 \leq m \leq \lceil kB\rceil - 1
\]
For any $\delta > 0$, if $0 \leq m \leq k(B-\delta)$, then the gap is $\geq
k\delta $. Taking $k \gg 1$ large enough, this gap can be made larger than $2\|V\|_{L^\infty}$, implying that the bands \eqref{equation:confinement} do not overlap.

The energy $E_k$ in \eqref{equation:pouet} satisfies $E_k \leq E_c-\eps$ for some $\eps > 0$. This implies that
\begin{equation}
\label{equation:ek}
E_k = k^{-2}\lambda_{k,m_k} + \mc{O}(k^{-2}),
\end{equation}
for some $0 \leq m_k \leq  k(B -\delta)$, where $\delta := \delta(\eps) > 0$ can be made explicit. To simplify notation, we write $m:=m_k$. 

We claim that:
\begin{equation}
\label{equation:ouf}
(\mathbbm{1}-\Pi_{k,m}) u_k = \mc{O}(k^{-1}).
\end{equation}
Here and in the remainder of the proof, the $\bigo$'s are measured in
$L^2$-norm. To prove \eqref{equation:ouf}, write (we use $(\Delta_{k} -\lambda_{k,m})\Pi_{k,m} =0$ in the first equality):
\[
\begin{split}
(\Delta_{k} -\lambda_{k,m})(\mathbbm{1}-\Pi_{k,m})u_k &= (\Delta_{k}
-\lambda_{k,m})u_k   \\
& \overset{\eqref{equation:ek}}{=} k^{2}(k^{-2}\Delta_k- E_k) u_k + \mc{O}(1) \\
& \overset{\eqref{equation:pouet}}{=} - V u_k + \mc{O}(1) \\
&= \mc{O}(1).
\end{split}
\]
However, the gaps between two consecutive eigenvalues of $\Delta_k$ is of size
$\geq k\delta $. Hence $(\Delta_{k} -\lambda_{k,m})$ can be inverted on
$\ran(\mathbbm{1}-\Pi_{k,m})$ and its inverse has norm $\mc{O}_{\mc{L}( L^2 )}(k^{-1})$ which implies \eqref{equation:ouf}.

We will now apply Lemma \ref{lemma:normal-form}. For that, consider $R \in \Psi^{\comp}_{\mathrm{tsc}}(\Sigma)$ with microsupport in $D^*\Sigma$ such that $R \equiv V$ microlocally in $\{p \leq E_c-\eps/2\}$ (for instance, consider $R = \Op_k(\chi \cdot \pi^*V)$, where $\pi : T^*\Sigma \to \Sigma$ is the footpoint projection, $\chi \in C^\infty_{\comp}(D^*\Sigma)$ is a cutoff function equal to $1$ on $\{p \leq E_c-\eps/2\}$). Since $u_k$ is microsupported in $\{p \leq E_c-\eps\}$ by \eqref{equation:pouet}, this implies that
\begin{gather} \label{eq:pouet_bis}
(k^{-2}\Delta_k + k^{-2}V) u_k = (k^{-2}\Delta_k + k^{-2}R) u_k +
\mc{O}(k^{-\infty}) = E_k u_k  .
\end{gather}
We then apply Lemma \ref{lemma:normal-form} with $\ell=4$, the function $\varphi$ such that $\varphi(\mathbf{A}_k) = k^{-2}\Delta_k$ and the perturbation $k^{-2}R$. We write
$U :=
e^{ih^2S_2}e^{ihS_1}e^{iS_0}$ so that
$$ U(k^{-2}\Delta_k + k^{-2}R)U^* = k^{-2}
\Delta_k + k^{-2}R' + \mc{O}(k^{-4}) $$
where $[R', \Delta_k]=0$, $R'$ is
selfadjoint, and the principal symbol of $R'$ is $\langle V \rangle$. 
Inserting $U$ in \eqref{eq:pouet_bis}, and writing $u_k' := U u_k$, we have
$$ (k^{-2} \Delta_k + k^{-2}R' ) u_k' = E_k u_k' +  \mc{O} ( k^{-4}) $$
 Applying $\Pi_{k,m}$ to the previous equation, and using that $[R',\Pi_{k,m}] = 0$ as $\Pi_{k,m}$ is a function of $\Delta_k$, we find
$$ (k^{-2} \Delta_k + k^{-2}R' ) \Pi_{k,m} u_k' = E_k \Pi_{k,m} u_k' +  \mc{O}
( k^{-4}) $$
Thus $$  k^{-2}R'  \Pi_{k,m} u_k' = (E_k - k^{-2} \lambda_{m,k} ) \Pi_{k,m} u_k' +  \mc{O}
( k^{-4}) $$
hence
$$ R' \Pi_{k,m}u_k'  = E'_k \Pi_{k,m} u_k' + \mc{O} (k^{-2}),
$$
where $E'_k = k^{2} ( E_k - k^{-2} \lambda_{k,m} ) = \mc{O}(1)$. 

% We apply Lemma \ref{lemma:normal-form} with $\ell=4$ and $\varphi$, the function such that $\varphi(\mathbf{A}_k) = k^{-2}\Delta_k$. We write $U := e^{ih^2S_2}e^{ihS_1}e^{iS_0}$. Inserting $U$ in \eqref{equation:pouet}, and writing $u_k' := U u_k$, we find:
% \[
%  U(k^{-2}\Delta_k + k^{-2}V)U^* u_k' = (k^{-2} \Delta_k + k^{-2}R + \mc{O}(k^{-4}))u_k' = (k^{-2}\lambda_{k,m}+\mc{O}(k^{-2}))u_k',
% \]
% where $[R,k^{-2}\Delta_k]=0$, $R$ is selfadjoint, and the principal symbol of $R$ is $\langle V \rangle$. Applying $\Pi_{k,m}$ to the previous equation, and noting that $[R,\Pi_{k,m}] = 0$ as $\Pi_{k,m}$ is a function of $\Delta_k$, we find
% \[
% \begin{split}
% (k^{-2} \Delta_k + k^{-2}R)\Pi_{k,m}u_k' & = (k^{-2}\lambda_{k,m} + k^{-2}R \Pi_{k,m})u_k' \\
% & = (k^{-2}\lambda_{k,m}+\mc{O}(k^{-2}))\Pi_{k,m} u_k' + \mc{O}_{L^2}(k^{-4}).
% \end{split}
% \]
% This results in
% \[
% R \Pi_{k,m}u_k'  = E'_k \Pi_{k,m} u_k' + \mc{O}_{L^2}(k^{-2}),
% \]
% where $E'_k = \mc{O}(1)$.
Applying Proposition \ref{proposition:scl-measure}, we find that any semiclassical measure $\mu$ associated to $\Pi_{k,m} u_k'$ is invariant by $H^{\Omega}_{\langle V \rangle}$. To conclude, it suffices to observe that $U = \mathbbm{1} +  \mc{O}_{\mc{L}(L^2)}(k^{-1})$, so $u_k = u'_k + \mc{O}_{L^2}(k^{-1})$; using \eqref{equation:ouf}, we find that $u_k = \Pi_{k,m} u'_k +\mc{O}(k^{-1})$ so the semiclassical measures of $\Pi_{k,m} u_k'$ and $u_k$ are the same. This proves the claim.
\end{proof}

\subsection{Proof of Theorem \ref{theorem:horocyclic}} The key ingredient in the proof of Theorem \ref{theorem:horocyclic} is a long time Egorov theorem which is made possible thanks to Lemma \ref{lemma:wow}. Ultimately, this relies on the fact that the horocyclic flow is parabolic (that is its differential grows at most linearly in time). We emphasize that the we did not try to optimize the remainder terms in what follows.

We suppose that
\begin{equation}
\label{equation:to-use}
 k^{-2}\Delta_k u_k = (E_c+\mc{O}(h^\ell))u_k, \qquad u_k \in C^\infty(\Sigma,L^{k}),\; \|u_k\|_{L^2(\Sigma,L^{k})} = 1,
 \end{equation}
 where $\ell > 0$ is fixed.
 
 Let $\chi \in C^\infty_{\comp}(\R)$ be a non negative bump function equal to $1$ on $[-1/2,1/2]$ and $0$ outside of $[-1,1]$. Let $0 < \alpha < \min(1/28,\ell)$ and define
 \[
 \chi_k(x,\xi) := \chi(k^\alpha (p(x,\xi)-E_c)).
 \]
 The following holds:
 
 \begin{lemma} \label{lemma:1} $u_k = \Op_k(\chi_k)u_k + \mc{O}_{L^2}(k^{-\min(\ell-\alpha,1-18\alpha)})$. \end{lemma}
 
% {\color{red}TL: Si on fait le choix de prendre $\alpha = \ell/40$ à la fin, alors dès que $\ell \leq 40/57$, on a $\min(\ell-\alpha,1-18\alpha) = \ell-\alpha$.}
 
 \begin{proof}
 Define $s(r) := (1-\chi(r))/r$ for $r \in \R$ and set $s_k := s(k^\alpha(p(x,\xi)-E_c))$, that is
 \[
 1-\chi_k = k^\alpha (\tfrac{1}{2}|\xi|^2-E_c) s_k.
 \]
Observe that all the derivatives of $s$ are bounded on $\R$. In addition, for all $n \geq 0$, there exists $C_n > 0$ such that:
 \begin{equation}
 \label{equation:sk}
 \|s_k\|_{C^n(T^*\Sigma)} \leq C_n k^{\alpha n}.
 \end{equation}
 We let $a \in C^\infty_{\comp}(T^*\Sigma)$ be a cutoff function such that $a \equiv 1$ on the energy layers $0 \leq E \leq 2 E_c$ and $a \equiv 0$ for $E > 3E_c$. Notice that $\Op_k(a)u_k = u_k + \mc{O}_{L^2}(k^{-\infty})$ as the eigenstates $u_k$ are microlocalized on the energy shell $\{p=E_c\}$. Equivalently, for all $b \in S^0(T^*\Sigma)$ such that $\supp(b) \cap \{p=E_c\}=\emptyset$, one has:
 \begin{equation}
 \label{equation:microloc}\Op_k(b)u_k = \mc{O}_{L^2}(k^{-\infty}).
 \end{equation}
 
 The following holds for some operator $R_k \in \Psi^1_{\mathrm{tsc}}(\Sigma)$:
 \[
 \begin{split}
 u_k &- \Op_k(\chi_k)u_k \\
 & = \Op_k(1-\chi_k) u_k  = k^\alpha \Op_k((\tfrac{1}{2}|\xi|^2-E_c)s_k) u_k  \\ &
 \overset{\eqref{equation:microloc}}{=} k^\alpha \Op_k(a(\tfrac{1}{2}|\xi|^2-E_c) \cdot a s_k) u_k + \mc{O}_{L^2}(k^{-\infty}) \\
 & \overset{\eqref{equation:product-fine}}{=} k^\alpha \left(\Op_k(a s_k) \Op_k(a (\tfrac{1}{2}|\xi|^2-E_c)) + \mc{O}_{\mc{L}(L^2)}(k^{-1}\|s_k\|_{C^{17}(T^*\Sigma)})\right) u_k +  \mc{O}_{L^2}(k^{-\infty}) \\
 & = k^{\alpha} \Op_k(a s_k) \left(\Op_k(a)(k^{-2}\Delta_k - E_c) + k^{-1} R_k\right) u_k + \mc{O}_{L^2}(k^{-1+18\alpha})  \\
 & \overset{\eqref{equation:to-use}}{=} k^{\alpha} \Op_k(as_k) \mc{O}_{L^2}(k^{-\ell}) + \mc{O}_{L^2}(k^{-1+\alpha} + k^{-1+18\alpha}) \\
 & \overset{\eqref{equation:calderon-fine}}{=} \mc{O}_{L^2}(k^{\alpha-\ell}(\|as_k\|_{L^\infty}+k^{-1/2+14\alpha}) +
 k^{-1+18\alpha})\\ & = \mc{O}_{L^2}(k^{-\min(\ell-\alpha,1-18\alpha)}),
 \end{split}
 \]
 where we used in the second line that $u_k$ is microlocalized on the energy
 shell $\{E = E_c\}$ (see \eqref{equation:microloc}), in the third line \eqref{equation:product-fine}
 (this requires the symbols to have compact support, which is why we have
 introduced $a$), in the last line $-1/2+14\alpha < 0$, and \eqref{equation:sk} at each step to bound the derivatives of $s_k$. Also notice that $R_k$ is a pseudodifferential operator of order $1$ but, using \eqref{equation:microloc}, we obtain:
 \[
 k^{-1} R_k u_k = k^{-1} R_k \Op_k(a) u_k = \mc{O}_{L^2}(k^{-1}),
 \]
 as $a$ has compact support in $T^*\Sigma$.
 \end{proof}
 
The states $(u_k)_{k \geq 0}$ microlocalize on the critical energy shell $S^*_c\Sigma = \{p=E_c\}$. We introduce the notation
 \[
 C_\delta := \{ E_c-\delta \leq p(x,\xi) \leq E_c + \delta\}.
 \]
In the following lemma, we need to keep track of the exact number of derivatives on the symbol $a$:
 
 \begin{lemma}
 \label{lemma:2}
Let $0 < \eps < 1/28$. There exists a constant $C > 0$ such that for all $a \in C^\infty_{\comp}(T^*\Sigma)$ with $\supp(a) \subset \{|\xi| \leq 10E_c\}$,
 \[
 |\langle\Op_k(a)u_k,u_k\rangle_{L^2}| \leq \|a\|_{L^\infty(S^*_c\Sigma)}  + \mc{O}(k^{-\eps} \|a\|_{C^1(T^*\Sigma)} + k^{-\min(\ell-\eps, 1/2-14\eps)}\|a\|_{C^{17}(T^*\Sigma)}).
 \]
 \end{lemma}
 
 %{\color{red}TL: Si on fait le choix de prendre $\eps = \ell/2$ à la fin, alors dès que $\ell \leq 1/15$, on a $\min(\ell-\eps,1/2-14\eps) = \ell-\eps=\ell/2$. Et de même $k^{-\eps} = k^{-\ell/2}$}
 
 \begin{proof}
 We first prove that
  \begin{equation}
  \label{equation:7}
 |\langle\Op_k(a)u_k,u_k\rangle_{L^2}| \leq \|a\|_{L^{\infty}(C_{k^{-\eps}})} + \mc{O}(k^{-\min(\ell-\eps, 1/2-14\eps)}\|a\|_{C^{17}(T^*\Sigma)}).
 \end{equation}
We set $\psi_k := \chi(k^\eps (p(x,\xi)-E_c))$ and start by decomposing:
 \[
 \begin{split}
 \langle \Op_k(a)u_k,u_k\rangle_{L^2} & = \underbrace{\langle \Op_k(a)\Op_k(\psi_k) u_k,u_k\rangle_{L^2}}_{\mathrm{(I)}} + \underbrace{\langle (1-\Op_k(\psi_k))u_k,\Op_k(a)^*u_k\rangle_{L^2}}_{\mathrm{(II)}}.
 \end{split}
 \]
 The second term $\mathrm{(II)}$ is bounded by:
 \[
 \begin{split}
|\langle (1-\Op_k(\psi_k))u_k,\Op_k(a)^*u_k\rangle_{L^2}| & \leq \|(1-\Op_k(\psi_k))u_k\|_{L^2} \|\Op_k(a)^*u_k\|_{L^2} \\
& \leq \|(1-\Op_k(\psi_k))u_k\|_{L^2} \|\Op_k(a)\|_{\mathcal{L}(L^2)} \\
& \leq C k^{-\min(\ell-\eps,1-18\eps)} (\|a\|_{L^\infty} + k^{-1/2}\|a\|_{C^{14}}) \\
& \leq Ck^{-\min(\ell-\eps,1-18\eps)}\|a\|_{C^{17}},
\end{split}
\]
where we used that $\|\Op_k(a)\|_{\mathcal{L}(L^2)} = \|\Op_k(a)^*\|_{\mathcal{L}(L^2)}$
in the second line, and Lemma \ref{lemma:1} together with
\eqref{equation:calderon-fine} in the third line.
% \footnote{en utilisant Calderon-Vaillancourt, on a tout de suite $\| a \|_{C^3(T^*\Si)}$ {\color{red}TL: Oui mais après on perd un 17 dans tous les cas pour borner l'autre terme}}.
The fourth line is a rougher estimate where the $\|a\|_{C^{17}(T^*\Sigma)}$ norm is involved as it will also appear in bounding $\mathrm{(I)}$.

As to the first term $\mathrm{(I)}$, we have:
\[
\begin{split}
|\langle \Op_k(a)\Op_k(\psi_k) u_k,u_k\rangle_{L^2}| &= |\langle \Op_k(a\psi_k)u_k,u_k\rangle_{L^2}| + \mc{O}(k^{-1}\|a\|_{C^{17}}\|\psi_k\|_{C^{17}}) \\
& \leq C(\|a\psi_k\|_{L^\infty} + k^{-1/2}\|a\psi_k\|_{C^{14}}) + \mc{O}(k^{-1+17\eps}\|a\|_{C^{17}}) \\
& \leq C \|a\psi_k\|_{L^\infty} + \mc{O}(\|a\|_{C^{14}}k^{-1/2+14\eps}+\|a\|_{C^{17}}k^{-1+17\eps}).
\end{split}
\]
Since $\eps < 1/28$, we find that $-1+17\eps < -1/2 + 14\eps < 0$ so
\[
\|a\|_{C^{14}(T^*\Sigma)}k^{-1/2+14\eps}+\|a\|_{C^{17}}k^{-1+17\eps} \leq C k^{-1/2+14\eps} \|a\|_{C^{17}}.
\]
Summing up $\mathrm{(I)}$ and $\mathrm{(II)}$, and using that $-1+18\eps < -1/2+14\eps < 0$, we find that
\[
\begin{split}
| \langle \Op_k(a)u_k,u_k\rangle_{L^2}| &\leq C \|a\psi_k\|_{L^\infty(T^*\Sigma)} + \mc{O}(k^{-\min(\ell-\eps, 1/2-14\eps)}\|a\|_{C^{17}(T^*\Sigma)}) \\
& \leq C \|a\|_{L^\infty(C_{k^{-\eps}})} + \mc{O}(k^{-\min(\ell-\eps, 1/2-14\eps)}\|a\|_{C^{17}(T^*\Sigma)}).
\end{split}
\]
This proves \eqref{equation:7}.

Now, observe that for any $\xi \in C_{k^{-\eps}}$, we can write
\[
\xi = \varphi_{t(\xi)}^{V_\perp}(\xi_c) = e^{t(\xi)}\xi_c,
\]
where $\xi_c \in S^*_c\Sigma$ lies on the critical energy shell, $|t(\xi)| \leq C k^{-\eps}$ and $(\varphi_t^{V_\perp})_{t \in \R}$ denotes the flow of $V_\perp$ on $T^*\Sigma$ (Euler vector field), see \eqref{equation:flot-vperp}. Hence:
\[
a(\xi) = a(\xi_c) + \int_0^{t(\xi)} V_\perp a(\varphi^{V_\perp}_t(\xi_c)) \dd t.
\]
This leads to
\[
\|a\|_{L^\infty(C_{k^{-\eps}})} \leq \|a\|_{L^\infty(S^*_c\Sigma)} + k^{-\eps}\|a\|_{C^1(T^*\Sigma)}.
\]
Inserting this in \eqref{equation:7} proves the claim.
 \end{proof}

% which is made possible thanks to Lemma \ref{lemma:bound-flot}:
%
%\begin{lemma}
%Let $\eps > 0$. Then for all $a \in C^\infty_{{\comp}}(T^*\Sigma)$, for all $t \leq h^{\eps-1/2}$, 
%\[
%e^{itk\Delta_k} \Op_k(a) e^{-itk\Delta_k} = \Op_k(a \circ h_t) + \mc{O}_{L^2 \to L^2}(h).
%\]
%{\color{red}XX condition sur le support de $a$}
%\end{lemma}
%
%\begin{proof}
%This follows from \cite[Proposition 2.8]{Bouzouina-Robert-02}. The result if formulated there for Hamiltonians with periodic Hamiltonian flow but, inspecting the proof, it only relies on the key estimate \eqref{equation:bound-flot} for the flow.
%\end{proof}

We now prove Theorem \ref{theorem:horocyclic}

\begin{proof}[Proof of Theorem \ref{theorem:horocyclic}]
Let $a \in C^\infty_{\comp}(T^*\Sigma)$ with $\supp(a) \subset \{|\xi| \leq 10\}$. We assume that the quasimodes $(u_k)_{k \geq 0}$ satisfy \eqref{equation:to-use} with $0 \leq \ell \leq 1/15$.  We consider
\[
\eps := \ell/2, \qquad \alpha := \eps/20 = \ell/40.
\]
Then
\[
\min(\ell-\alpha,1-18\alpha) = \ell-\alpha = 39\ell/40, \qquad \min(\ell-\eps,1/2-14\eps)=\ell-\eps =\ell/2 < 39\ell/40. 
\]
In Lemmas \ref{lemma:1} and \ref{lemma:2}, the remainder term is therefore bounded respectively by $\mc{O}_{L^2}(k^{-\ell/2})$ and $\mc{O}(k^{-\ell/2}\|a\|_{C^{17}(T^*\Sigma)})$.

 We have:
\begin{equation}
\label{equation:la1}
\begin{split}
\langle \Op_k&(a)u_k,u_k\rangle_{L^2} \\
&= \langle \Op_k(a)(\Op_k(\chi_k) u_k + \mc{O}_{L^2}(k^{-\ell/2})),u_k\rangle_{L^2} \\
& \overset{\eqref{equation:calderon-fine}}{=} \langle \Op_k(a)\Op_k(\chi_k) u_k,u_k\rangle_{L^2} + \mc{O}\left(k^{-\ell/2}(\|a\|_{L^\infty} + k^{-1/2}\|a\|_{C^{14}})\right) \\
& \overset{\eqref{equation:product-fine}}{=} \langle \Op_k(a \chi_k) u_k,u_k\rangle_{L^2} + \mc{O}\left(k^{-1}\|a\|_{C^{17}}\|\chi_k\|_{C^{17}} + k^{-\ell/2}\|a\|_{L^\infty} + k^{-1/2-\ell/2}\|a\|_{C^{14}}\right) \\
& = \langle \Op_k(a \chi_k) u_k,u_k\rangle_{L^2} + \mc{O}\left(k^{-1+17\ell/40}\|a\|_{C^{17}} + k^{-\ell/2}\|a\|_{L^\infty} + k^{-1/2-\ell/2}\|a\|_{C^{14}}\right) \\
& = \langle \Op_k(a \chi_k) u_k,u_k\rangle_{L^2} + \mc{O}(k^{-\ell/2}\|a\|_{C^{17}}),
\end{split}
\end{equation}
where we used Lemma \ref{lemma:1} in the first line and $\|\chi_k\|_{C^{17}} \leq C k^{17\alpha} = Ck^{17\ell/40}$ in the fourth line.

Since $u_k$ are eigenmodes of $\Delta_k$, we may now write for $t \geq 0$:
\[
\begin{split}
\langle  \Op_k(a \chi_k) u_k,u_k\rangle_{L^2} &= \langle e^{itk^{-1}\Delta_k} \Op_k(a \chi_k) e^{-itk^{-1}\Delta_k}u_k,u_k\rangle_{L^2} \\
& \overset{\eqref{equation:egorov-fine}}{=} \langle \Op_k((a \chi_k) \circ \Phi_t) u_k, u_k \rangle_{L^2} + \mc{O}\left(k^{-1} \int_0^t \|(a \chi_k) \circ \Phi_s)\|_{C^{17}} \dd s\right).
\end{split}
\]
Applying \eqref{equation:key-bound} with $n=17$, and using that $\chi_k a$ is localized on the energy layer $E \leq E_c + k^{-\alpha} = E_c + k^{-\ell/40}$, we find that the remainder term in the previous equation is bounded by
\begin{equation}
\label{equation:dingo}
\leq C k^{-1} \int_0^t \langle s \rangle^{m_{17}} e^{17 \sqrt{2}k^{-\ell/80} s} \|a \chi_k\|_{C^{17}}  \dd s \leq C k^{-1+17\ell/40} \langle t\rangle^{m_{17}+1} e^{17\sqrt{2}k^{-\ell/80} t} \|a\|_{C^{17}}.
\end{equation}
One has $m_{17} = 3 \times 17+ 17\times 18/2 = 204$. Hence, we find that:
\begin{equation}
\label{equation:pfiou}
\langle  \Op_k(a \chi_k) u_k,u_k\rangle_{L^2} = \langle \Op_k((a \chi_k) \circ \Phi_t) u_k, u_k \rangle_{L^2} + \mc{O}(k^{-1+17\ell/40} \langle t\rangle^{205} e^{17\sqrt{2}k^{-\ell/80} t} \|a\|_{C^{17}}).
\end{equation}
For a function $f \in C^\infty(T^*\Sigma)$, define the ergodic average for $T \geq 0$:
\[
\langle f \rangle_T := \dfrac{1}{T} \int_0^T f \circ \Phi_t ~\dd t.
\]
Averaging \eqref{equation:pfiou}, we find that for $T \geq 0$:
\[
\langle  \Op_k(a \chi_k) u_k,u_k\rangle_{L^2} = \langle \Op_k(\langle a \chi_k\rangle_T) u_k, u_k \rangle_{L^2} + \mc{O}(k^{-1+17\ell/40} \langle T\rangle^{205} e^{17\sqrt{2}k^{-\ell/80} T} \|a\|_{C^{17}}).
\]
Write
\[
A := \int_{S^*_c\Sigma} a~ \dd \mu_{\mathrm{Liouv}}.
\]
We then obtain:
\[
\begin{split}
|\langle&  \Op_k(a \chi_k) u_k,u_k\rangle_{L^2} - A| \\
&= |\langle \Op_k\left(\langle a \chi_k- A\rangle_T\right) u_k, u_k \rangle_{L^2}|  + \mc{O}(k^{-1+17\ell/40} \langle T\rangle^{205} e^{17\sqrt{2}k^{-\ell/80} T} \|a\|_{C^{17}}) \\
& \lesssim \|\langle a \chi_k-A\rangle_T\|_{L^\infty(S^*_c\Sigma)} + k^{-\ell/2} \|\langle a \chi_k-A\rangle_T\|_{C^{1}} +  k^{-\ell/2}\|\langle a \chi_k-A\rangle_T\|_{C^{17}}\\
& \hspace{3cm}  + k^{-1+17\ell/40} \langle T\rangle^{205} e^{17\sqrt{2}k^{-\ell/80} T} \|a\|_{C^{17}} \\
& \lesssim  \|\langle a \chi_k-A\rangle_T\|_{L^\infty(S^*_c\Sigma)} +  k^{-\ell/2}\|\langle a \chi_k-A\rangle_T\|_{C^{17}} + k^{-1+17\ell/40} \langle T\rangle^{205} e^{17\sqrt{2}k^{-\ell/80} T} \|a\|_{C^{17}},
\end{split}
\]
where we have applied Lemma \ref{lemma:2} in the second inequality (here $\eps = \ell/2$ and $\ell-\eps=\ell/2$). Notice that $\chi_k \equiv 1$ on $C_{k^{-\ell/2}}$ by construction, and $\chi_k \circ \Phi_t = \chi_k$ as $(\Phi_t)_{t \in \R}$ preserves the energy layers $\{p=E\}$. As a consequence,
\[
\|\langle a \chi_k-A\rangle_T\|_{L^\infty(S^*_c\Sigma)} = \|\langle a-A\rangle_T\|_{L^\infty(S^*_c\Sigma)}.
\]
In addition, using the same bound as in \eqref{equation:dingo} (based on \eqref{equation:key-bound}), we may estimate:
\[
\|\langle a \chi_k-A\rangle_T\|_{C^{17}} \leq C k^{17\ell/40}\langle T\rangle^{205}e^{17\sqrt{2} k^{-\ell/80} T}\|a\|_{C^{17}}.
\]
(Notice that, technically, the function $\langle a \chi_k-A\rangle_T = \langle a\chi_k\rangle_T - A$ does not have shrinking support but $A$ is a constant, so it vanishes when computing higher order derivatives of this function and thus the support of the derivatives is shrinking, leading to the same bound.) This leads to:
%
%
%\[
%\begin{split}
%|\langle  \Op_k&(a \chi_k) u_k,u_k\rangle_{L^2} - A| \\
%& \lesssim \|\langle a-A\rangle_T\|_{L^\infty(C_{k^{-\ell/2}})} + (k^{-\ell/2+17\ell/40} + k^{-1+17\ell/40})  \langle T\rangle^{205} e^{17\sqrt{2}k^{-\ell/80} T} \|a\|_{C^{17}} \\
%& \lesssim \|\langle a-A\rangle_T\|_{L^\infty(C_{k^{-\ell/2}})} + k^{-3\ell/40} \langle T\rangle^{205} e^{17\sqrt{2}k^{-\ell/80} T} \|a\|_{C^{17}}.
%\end{split}
%\]
%
%
%
%This implies that:
%\[
%\begin{split}
%\|\langle a-A\rangle_T\|_{L^\infty(C_{k^{-\ell/2}})} & \leq \|\langle a-A\rangle_T\|_{L^\infty(S^*_c\Sigma)} + \mc{O}(k^{-\ell/2}\|\langle a-A\rangle_T\|_{C^1(C_{k^{-\ell/2}})}) \\
%& \leq \|\langle a-A\rangle_T\|_{L^\infty(S^*_c\Sigma)} + \mc{O}(k^{-\ell/2}\langle T\rangle^4 e^{\sqrt{2}k^{-\ell/4}T}\|a\|_{C^{1}}),
%\end{split}
%\]
%where we used \eqref{equation:key-bound} with $n=1$ in the second line applied to the restriction of $\langle a-A\rangle_T$ to $C_{k^{-\ell/2}}$. Combining the previous estimates leads to:

\begin{equation}
\label{equation:la2}
|\langle  \Op_k(a \chi_k) u_k,u_k\rangle_{L^2} - A| \lesssim  \|\langle a-A\rangle_T\|_{L^\infty(S^*_c\Sigma)} +  k^{-3\ell/40} \langle T\rangle^{205} e^{17\sqrt{2}k^{-\ell/80} T} \|a\|_{C^{17}}.
\end{equation}

Finally, applying Theorem \ref{lemma:horocycle}, we obtain that
\begin{equation}
\label{equation:la3}
 \|\langle a-A\rangle_T\|_{L^\infty(S^*_c\Sigma)} \lesssim \|a\|_{H^{3}(S^*_c\Sigma)} T^{-\theta}.
 \end{equation}
%and in the third line that
%\[
%\|\chi_k(a \circ \Phi_s)\|_{C^{17}} \leq C \|\chi_k\|_{C^{17}} \|a\circ\Phi_s\|_{C^{17}} \leq C k^{17\ell/4}\|a\circ\Phi_s\|_{C^{17}}.
%\]
Combining \eqref{equation:la1}, \eqref{equation:la2} and \eqref{equation:la3}, we find that for all $T \geq 0$:
\[
\langle \Op_k(a)u_k,u_k\rangle_{L^2} = \int_{S^*_c\Sigma} a \dd \mu + \mc{O}\left(k^{-3\ell/40} \langle T\rangle^{205} e^{17\sqrt{2}k^{-\ell/80} T} \|a\|_{C^{17}} + \|a\|_{H^{3}(S^*_c\Sigma)} T^{-\theta}\right).
\]
We then set
\[
\delta := \dfrac{3 \ell}{2 \cdot 40 \cdot 205} = \dfrac{\ell}{4100}.
\]
This guarantees that $205 \delta-3\ell/40 =-3\ell/80 < 0$. We apply the previous estimate with $T = k^{\delta}$. We then find that:
\[
\begin{split}
\langle \Op_k(a)u_k,u_k\rangle_{L^2} &= \int_{S^*_c\Sigma} a \dd \mu + \mc{O}(k^{-3\ell/80} \|a\|_{C^{17}(T^*\Sigma)} + k^{-\theta \ell/4100}\|a\|_{H^3(S^*_c\Sigma)}) \\
& = \int_{S^*_c\Sigma} a \dd \mu + \mc{O}(k^{-\theta \ell/4100}\|a\|_{C^{17}(T^*\Sigma)}),
\end{split}
\]
where we used that $\theta \ell/4100 < 3\ell/80$ as $\theta < 1/2$. This completes the proof.
\end{proof}

\bibliographystyle{alpha}
%\nocite{*}
\bibliography{Biblio}

\end{document}